\documentclass[12pt]{article}

\usepackage{mathtools,amssymb,amsmath,amsfonts,mathrsfs,textcomp,amsthm,latexsym}
\usepackage{tikz,textcomp}
\usepackage{enumitem}
\usepackage{geometry}
\usepackage{cite}
\usepackage{booktabs,enumitem,microtype}
\usepackage[hidelinks]{hyperref}
\usepackage{bm}
\usepackage{comment}
\usepackage{listings}
\usepackage[hidelinks]{hyperref}
\newcommand{\Ex}{\mathrm{ex}} 
\newcommand{\N}{\mathcal{N}} 
\newcommand{\K}{\mathcal K}
\newcommand{\Aut}{\operatorname{Aut}}
\newcommand{\EE}{\mathbb E}
\newtheoremstyle{compact}{4pt}{4pt}{\itshape}{}{\bfseries}{.}{0.5em}{}
\theoremstyle{compact}
\newtheorem{lem}{Lemma}[section]
\newtheorem{thm}{Theorem}[section]
\newtheorem{cor}[thm]{Corollary}
\newtheorem{prop}[thm]{Proposition}

\newtheorem{clm}{Claim}
\newtheorem{remark}{Remark}

\allowdisplaybreaks  
\title{Supersaturation thresholds for expanded color-critical graphs \thanks{ E-mail: 1519036064@qq.com (L. Zhang), qliu@fzu.edu.cn (Q. Liu),  yhong@fzu.edu.cn (Y. Hong),
		 hujinze@fzu.edu.cn (J. Hu),
		 wangxi123@nenu.edu.cn (X. Wang).
		 The authors would like to thank  Wang Xi for validating the R program for us.
}}
\author{Liping Zhang$^{a}$, Qinghai Liu$^{a,c}$\footnote{Corresponding author.},  Yanmei Hong$^{b,c}$, 	Xi Wang$^{d}$, Jinze Hu$^{a}$ 
\\[1mm]
{\small \it
$^a$Center for Discrete Mathematics, Fuzhou University, Fuzhou, Fujian, 350108 China}\\
{\small \it
$^b$School of Mathematics and Statistics, Fuzhou University, Fuzhou, Fujian, 350108 China}\\
{\small \it
$^c$Key Laboratory for Operations Research and Cybernetics of Fujian Universities}\\
{\small \it
$^d$School of Mathematics and Statistics, Northeast Normal   Universities,}\\
{\small \it
 Changchun, Jilin, 130024 China}
}
\date{}
\begin{document} \maketitle
\begin{abstract}
For a $k$-uniform hypergraph $L$, let $h_L(n,q)$ be the minimum number of copies of $L$ in an $n$-vertex $k$-graph with $\mathrm{ex}(n,L)+q$ edges.
Let $t_L(n,q)$ be the minimum number of copies obtained by adding $q$ edges to an extremal $L$-free $k$-graph.
A graph is color-critical if deleting a suitable edge decreases its chromatic number. Let $\mathcal K_k(F)$ be the expanded graph obtained by adding $k-2$ new vertices to each edge of $F$, with disjoint sets of new vertices for distinct edges.

Mubayi proved that,   for  every clique $K_r$ with $r\ge4$ and  fixed  $ q>0$,
$
 h_{\mathcal K_3(K_r)}(n,q)\ge qt_{\mathcal K_3(K_r)}(n,1)$ for every  sufficiently large $n$.
Pikhurko and Yilma later proved that,  for  every color-critical graph $F$, there is a constant $\epsilon_F>0$ such that
$h_{F}(n,q)=t_{F}(n,q)$ for every  sufficiently large $n$ and  $1\le q\le\epsilon_F n$.
We extend this equality to expanded color-critical graphs.
Let $d>k\ge3$. For  every color-critical graph $F$ with $\chi(F)=d+1$,
  there is a constant $\epsilon>0$ such that
$h_{\mathcal K_k(F)}(n,q)=t_{\mathcal K_k(F)}(n,q)$ for every  sufficiently large $n$ and 
$1\le q\le\epsilon n^{k-1}$.
Let $q_L(n)$ be the largest integer $q$ such that
$h_L(n,s)=t_L(n,s)$ for every integer $1\le s\le q$.
We give a sufficient condition for
$q_L(n)=\Theta_L(n^{k-1})$ and an explicit family satisfying it.
We also give a different sufficient condition under which
$q_L(n)=\Omega_L(n^{k-1/2})$ and an explicit family satisfying it. 

\noindent
	\emph{Keywords.}
	Supersaturation; extremal graph; color-critical graph; chromatic number; expanded graph; hypergraph.
\end{abstract}
\section{Introduction}
Let $F$ be a fixed $k$-graph. The Tur{\'a}n \cite{turan1941} function $\mathrm{ex}(n,F)$, introduced by Tur{\'a}n, is one of the central objects in extremal theory. It is defined as the maximum number of edges in an $n$-vertex $F$-free $k$-graph. Determining
	$\mathrm{ex}(n,F)$, together with the structure of the corresponding
	extremal graphs, is a central problem in extremal
	combinatorics; see, for example,
\cite{GerbnerPalmer,GerbnerNagyVizer,MubayiVerstraete,Keevash}
	and the references therein.
	A natural refinement of the Tur{\'a}n problem is the supersaturation problem. Instead of forbidding a copy of $F$, one asks how many copies of $F$  appear once the number of edges exceeds the Tur{\'a}n threshold. More precisely, for a fixed $k$-graph $F$, let $h_F(n,q)$ denote the minimum number of copies of $F$ among all $n$-vertex $k$-graphs with $\mathrm{ex}(n,F)+q$ edges. 
   These copies are unlabelled and need not be induced. Determining the behavior of $h_F(n,q)$ is a central problem in extremal theory.
	
	The study of supersaturation can be traced back to a classical theorem of Rademacher~\cite{Rademacher1941}, who proved that every graph with $\mathrm{ex}(n,K_3)+1$ edges contains at least $\lfloor n/2\rfloor$ triangles (that is $h_{K_3}(n,1)\ge \lfloor n/2\rfloor$). 
	Shortly afterwards, Erd\H{o}s~\cite{Erdos1955}
	proved that
	$h_{K_3}(n,q)=q\lfloor n/2\rfloor$
	for $q\le 3$. 
	In fact, he found that one possible construction for graphs with $h_{K_3}(n,q)$ copies of $K_3$ is obtained by adding $q$
	edges to  the Tur\'an graph for $K_3$. Denote by $t_F(n,q)$ the smallest number of $F$-copies that can be achieved this way, and let $c(n,F):=t_F(n,1)$.
	By definition, $h_F(n,q)\le t_F(n,q)$ and $q\cdot c(n,F)\le t_F(n,q)$.
	More generally, he established analogous results for complete graphs, proving that for every fixed $r\ge3$, there exists a constant $c_d>0$ such that
	$h_{K_d}(n,q)=t_{K_d}(n,q)$
	whenever $q\le c_r n$. Later, Lov\'asz and Simonovits~\cite{LovaszSimonovits1983} determined the optimal value of $c_d$ for sufficiently large $n$, thereby resolving a longstanding conjecture of Erd\H{o}s. Their work essentially settled the supersaturation problem for cliques in the case $q = o(n^2)$.
	The supersaturation problem for cliques beyond the linear range proved substantially more difficult. After a sequence of partial results, asymptotic solutions were obtained by Razborov~\cite{Razborov2008} for $K_3$, Nikiforov~\cite{Nikiforov2011} for $K_4$, and Reiher~\cite{Reiher2016} for general $K_d$.
	
	Motivated by these developments, Mubayi~\cite{Mubayi2010} developed a systematic approach for supersaturation of color-critical graphs. Recall that a graph $F$ is called \emph{color-critical} if there exists an edge $e\in E(F)$ such that
	$\chi(F-e)<\chi(F)$.
	He proved that for every color-critical graph $F$, there exists a constant $c_F>0$ such that $h_F(n,q)\ge q\cdot c(n,F)$
	for all $1\le q\le c_F n.$
	This result unified and substantially generalized many earlier supersaturation results
	Following the approach developed in \cite{Mubayi2010}, Mubayi~\cite{Mubayi2013} extended the supersaturation theory to hypergraphs by combining Stability Theorem with the Hypergraph Removal Lemma. 
	Given a graph $F$, let $\mathcal K_k(F)$ denote the $k$-graph obtained by enlarging each edge of $F$ with $k-2$ new vertices. 
	In particular, he proved:
	\begin{thm}[Mubayi]
Fix  $r\ge4, q>0$.  For sufficiently large $n$,
$
 h_{\mathcal K_3(K_r)}(n,q)\ge qc(n,\mathcal K_3(K_r))$.  
	\end{thm} 
	More generally, the paper established supersaturation results for several important  $3$-graphs, including the Fano plane, cancellative $3$-graphs (that is,  if $A\cup B = A \cup C$ we have
	$B = C$), $3$-graphs contains no copy
	of $B_5 = {123,124,125,345}$, and some $4$-graphs, including the  books, $C_3 = {1234,3456,1256}$.  This work provided one of the first systematic treatments of supersaturation phenomena in  hypergraphs. 
	
	Our primary focus is on expanded  color-critical graphs, the exact supersaturation problem for expanded  color-critical graphs  remained widely open, it was not known whether the natural construction obtained by adding edges to an extremal $L$-free hypergraph is optimal. A major breakthrough in this direction was obtained by Pikhurko and Yilma~\cite{PikhurkoYilma2017}. They proved:
	\begin{thm}[Pikhurko, Yilma]
		For every color-critical graph $F$, there exists a constant $\varepsilon_F>0$ such that
		$
		h_F(n,q)=t_F(n,q)
		$
		for all
		$
		1\le q\le \varepsilon_F n
		$
		and sufficiently large $n$.     
	\end{thm}
To describe more precisely the range
for which the extremal construction remains optimal, they defined
\[
q_F(n):=
\max\left\{
q\in\mathbb N:
h_F(n,q')=t_F(n,q')
\text{ for all }q'\le q
\right\},
\]
and
$
c_1(F):=
\liminf_{n\to\infty}\frac{q_F(n)}{n}$.
Thus, $q_F(n)$ measures the largest initial range of $q$ for which
adding $q$ edges to an extremal graph remains optimal for minimizing
the number of copies of $F$.
Pikhurko and Yilma further determined $c_1(F)$ exactly for several
 color-critical graphs. In particular, 
$
c_1(F)=\frac12
$ for every odd cycle $F$, $
c_1(K_r)=\frac1{r-1}
$.
Moreover, for $r\ge2$ and $F=K_{r+2}-e$, where $K_{r+2}-e$ is
obtained from $K_{r+2}$ by removing one edge, they proved that
$
c_1(F)=\frac{r-1}{r^2}.
$

Our main result extends this threshold phenomenon from
color-critical graphs to expanded hypergraphs.
 Define $$\mathcal{F}(k,d)=\{\mathcal{K}^{(k)}(F),\ F \text{ is a color-critical graph with}\ \chi(F)=d+1\}.$$

We first estimate  $t_L(n,q)$ in a wider
range. 

\begin{thm}\label{thm-tlnq1}
Let  $d > k \ge 3$ be integers, and let $L\in \mathcal{F}(k,d)$ has $l$ vertices.
There is a constant $\delta_0>0$ such that,
for all sufficiently large $n$ and $1\le q\le\delta_0n^k$,
\[
qc(n,L)\le t_L(n,q)\le qc(n,L)+O_L(q^2n^{l-2k}).
\]
Consequently, $t_L(n,q)=(1+o(1))qc(n,L)$ whenever $q=o(n^k)$.
\end{thm}

\begin{thm}\label{thm-ht}\label{thm1}
Let  $d > k \ge 3$ be integers, and let $L\in \mathcal{F}(k,d)$. 
There is a constant $\delta_L>0$ such that  for all sufficiently large $n$ and 
$1\le q\le\delta_L n^{k-1}$,
$h_L(n,q)=t_L(n,q)$.
Moreover,  if $\delta_L<1/(2d^{k-1})$ then $h_L(n,q)=t_L(n,q)=qc(n,L).$
In particular, $q_L(n)=\Omega_L(n^{k-1})$.
\end{thm}

We next give a sufficient condition for an  upper bound of $q_L(n)$.
By Theorem~\ref{thm-c}, there is a constant $\alpha_L>0$ such that
$
C(n,L)=\alpha_Ln^{l-k}+O_L(n^{l-k-1})$
where $l=|V(L)|$.
The parameter $\zeta_L$ measures the first-order change in the
 single-edge count when vertices move between parts.
Its precise definition is given in Subsection~\ref{s2.1}.

\begin{thm}\label{thm-anchor-upper}
Let  $d > k \ge 3$ be integers, and let $L\in \mathcal{F}(k,d)$.  Suppose
$
|\zeta_L|>\frac{4d(d-k+1)}{k-2}\alpha_L$, then there are
constants $0< \theta_1<\theta_2$ such that,
for all sufficiently large $n$ and  $q=\lfloor\theta n^{k-1}\rfloor$ with $\theta_1<\theta<\theta_2$, 
\[
h_L(n,q)<qc(n,L)\le t_L(n,q).
\]
In particular, $q_L(n)=\Theta_L(n^{k-1})$.
\end{thm}

Let $F_{s,t,d}$ be obtained from the complete
$d$-partite graph with part sizes $s,t,\ldots,t$ by adding an edge 
in the part of size $s$. Applying  Theorem \ref{thm-anchor-upper} to $F_{s,t,d}$ gives corollary \ref{cor-multipartite-threshold}. 

\begin{cor}\label{cor-multipartite-threshold}\label{lem:critical}
Let $d>k\ge3$ and $s,t\ge2$  be integers.
Then $ F_{s,t,d}$ has chromatic number $d+1$, and the added edge is its unique
critical edge. Moreover,
\begin{equation}\label{eq-multipartite-zeta}
\zeta_{\mathcal K_k(F_{s,t,d})}=d\alpha_{\mathcal K_k(F_{s,t,d})} \bigl((s-t)(1-(k-2)t)-2\bigr).
\end{equation}
Thus, if
$\bigl|(s-t)(1-(k-2)t)-2\bigr|>\frac{4(d-k+1)}{k-2},$
then $q_{\mathcal K_k(F_{s,t,d})}(n)=\Theta_L(n^{k-1})$.
\end{cor}

Thus the order $n^{k-1}$ in Theorem~\ref{thm-ht} is sharp
for every fixed  $d>k\ge3$.
For example, $L=\mathcal K_3(F_{8,3,4})$ satisfies
$q_L(n)=\Theta(n^2)$.

We also give a sufficient condition for a larger upper bound of $q_L(n)$.
It combines the condition $\zeta_L=0$ with a condition on
copies through one prescribed vertex.
The notation in the following theorem is defined in Section~\ref{s5}.

\begin{thm}\label{thm-stationary-link}\label{thm:criterion}
Let  $d > k \ge 3$ be integers, and let $L\in \mathcal{F}(k,d)$.
Suppose $\zeta_L=0$ and for every link vector $\boldsymbol w$,
\begin{equation}\label{eq-link-rigidity}
R_L(\boldsymbol w)\le\alpha_L\max\{\eta(\boldsymbol w),0\}
\quad\Longrightarrow\quad
\min_{1\le i\le d}p_i(\boldsymbol w)=0.
\end{equation}
There is a constant $\delta_L>0$ such that, for all sufficiently large $n$ and $1\le q\le\delta_Ln^{k-1/2}$,
$h_L(n,q)=t_L(n,q)$. In particular,
$q_L(n)=\Omega_L(n^{k-1/2})$.
\end{thm}

Our first graph group gives a superquadratic lower bound.
\begin{thm}\label{thm:main}
Let $4\le d\le 25$ and $L_d=\mathcal K_3(F_{2,3,d})$,  then $q_{L_d}(n)=\Omega(n^{5/2})$. In particular,
there is a constant $\delta_d>0$ such that the following holds
for all sufficiently large $n$.
If $1\le q\le\delta_d n^{5/2}$, then every $n$-vertex $3$-graph
with $\text{ex}(n,L_d)+q$ edges that has $h_{L_d}(n,q)$ copies
of $L_d$ contains $T^{(3)}(n,d)$.
\end{thm}

The result gives a lower bound on the threshold.
Its exact order remains open.
Our  explicit family gives the
general quadratic lower bound.

\begin{thm}\label{thm:quadratic}
For every fixed integer $d\ge4$,
$q_{\mathcal K_3(F_{2,5,d})}(n)=\Theta(n^2).$ In particular, there is a constant $A_d>0$ such that the following holds
for all sufficiently large $n$ and $q=o(n^3)$.
If $ A_d n^2\le q$ then
\[ h_{\mathcal K_3(F_{2,5,d})}(n,q)<q\cdot t_{\mathcal K_3(F_{2,5,d})}(n,1)\le t_{\mathcal K_3(F_{2,5,d})}(n,q).\]
\end{thm}

For each $4\le d\le 25$, both $F_{2,3,d}$ and $F_{2,5,d}$ are
connected and have chromatic number $d+1$.
Their expansions are connected $3$-graphs.
The two theorems imply
\[
\frac{q_{\K_3(F_{2,3,d})}(n)}
     {q_{\K_3(F_{2,5,d})}(n)}
=
\Omega(n^{1/2}).
\]
Thus the uniformity and the chromatic number of the original
graph do not determine a common threshold order.

We use the following notation.
For a positive integer $m$, write $[m]=\{1,\ldots,m\}$.
For a finite set $X$, let $\binom{X}{r}$ be the family of its
$r$-element subsets.
For a graph or hypergraph $H$, write $V(H)$ and $E(H)$ for its
vertex set and edge set.
Set $|H|=|V(H)|$ and $e(H)=|E(H)|$.
For $U\subseteq V(H)$, let $H[U]$ be the subgraph induced by $U$,
and write $H-U$ for the deletion of the vertices in $U$.
For $B\subseteq E(H)$, write $H-B$ for the deletion of the edges
in $B$. Let $N_L(H)$ denote the number of unlabelled copies of $L$ in $H$.
These copies need not be induced.
 Most counts first use injective maps from a labelled copy of $L$.
Dividing by $|\operatorname{Aut}(L)|$ gives the unlabelled count,
where $\operatorname{Aut}(L)$ is the automorphism group of $L$.

For pairwise disjoint vertex sets $U_1,\ldots,U_d$, let
$K=K^{(k)}(U_1,\ldots,U_d)$ be the complete $d$-partite $k$-graph
with vertex set $U_1\cup\cdots\cup U_d$.
Its edges are all $k$-sets that meet each $U_i$ in at most one
vertex.
Such $k$-sets are called \emph{crossing} with respect to this partition.

We use the standard asymptotic notation.
Suppose that $f(x)>0$ for all sufficiently large $x$.
Denote $g(x)=O(f(x))$ if $|g(x)|\le Cf(x)$ for some constant
$C>0$ and all sufficiently large $x$.
Denote $g(x)=\Omega(f(x))$ if $g(x)\ge cf(x)$ for some constant
$c>0$ and all sufficiently large $x$.
Denote $g(x)=\Theta(f(x))$ if
$cf(x)\le g(x)\le Cf(x)$ for some constants $c,C>0$ and all
sufficiently large $x$.
Finally, $g(x)=o(f(x))$ means that $g(x)/f(x)\to0$ as $x\to\infty$.
Unless stated otherwise, all asymptotic statements are as
$n\to\infty$, with $L$, $k$, and $d$ fixed.
We write $a\ll b$ to mean that $a>0$ is chosen sufficiently small
than $b$ and the fixed parameters.

Section~\ref{s2} develops the tools for the main proofs:
Subsection~\ref{s2.1} gives the extremal and counting results and defines
$\zeta_L$.
Subsection~\ref{s2.2} establishes stability and local structural estimates.
Subsection~\ref{s2.3} treats balancing and gives an unbalanced construction.
Section~\ref{s3} proves Theorems~\ref{thm-tlnq1} and~\ref{thm-ht}.
Section~\ref{s4} proves  Theorem~\ref{thm-anchor-upper},  followed by the examples in 
Corollary~\ref{cor-multipartite-threshold}.
Section~\ref{s5} combines counting through a fixed vertex with the
stationary link condition to prove Theorem~\ref{thm-stationary-link}. Section~\ref{s6} gives two precise graph groups, each of which satisfies two sufficient conditions.
The appendix proves the extremal theorem stated in Theorem~\ref{extremal-graph}.

\section{Preliminary Results}\label{s2}
This section collects the tools needed for the main proofs.
First, we count copies in complete multipartite hypergraphs with
added edges. Second, 
we  establish the structural and balancing estimates.

\subsection{Extremal and   counting}\label{s2.1}
We first state the extremal theorem and count copies with prescribed
non-crossing edges.
We then study how these counts change with the part sizes and
define $\zeta_L$.

Let $T^{(k)}(n,d)$ denote the complete $d$-partite $k$-graph on $n$ vertices, where each partition class has size $\lfloor n/d\rfloor$ or $\lceil n/d\rceil$, and each edge intersects every partition class in at most one vertex. Let $t^{(k)}(n,d):=e(T^{(k)}(n,d))$.
	The following theorem follows from a result of Alon and  Pikhurko     (personal communication, 2007). For completeness, we include their proof in Section \ref{A}.	\begin{thm}[Alon, Pikhurko~\cite{AlonPikhurko}]\label{extremal-graph}
		Let $d > k \ge 3$ be integers, and let $L\in \mathcal{F}(k,d)$. For all sufficiently large \(n\),  $T^{(k)}(n,d)$ is the unique $n$-vertex extremal $k$-graph for $L$.
	\end{thm}

 

Fix $K=K^{(k)}(U_1,\ldots,U_d)$ with $u_i=|U_i|$ for each $i\in[d]$, and let $\boldsymbol u=(u_1,\dots,u_d)$. For  a non-crossing edge $e\in{V(K)\choose k}$, define the  \emph{part-pattern} of $e$ with respect to
$(U_1,\ldots,U_d)$ by
\[\rho_U(e):=
\bigl(|e\cap U_1|,\ldots,|e\cap U_d|\bigr).
\]Observe that $\rho_U(e)$ is an ordered tuple, Permuting the  part labels permutes its coordinates.
Let
\[
\mathcal P
:=
\left\{
\boldsymbol a=(a_1,\ldots,a_d)\in\mathbb Z_{\ge0}^d:
\sum_{i=1}^d a_i=k,\ 
\max_{i\in[d]}a_i\ge2
\right\}.
\]
Thus $\mathcal P$ is precisely the set of possible part-patterns
of non-crossing $k$-sets. Since $k$ and $d$ are fixed,
$|\mathcal P|$ is a constant independent of $n$. 

For a fixed non-crossing
edge set $S$, let $C(\boldsymbol u,S;L)$ (or $C(u_1,\dots,u_d,S;L)$) count copies in
$K\cup S$ whose exact non-crossing edge set is $S$. In balanced parts $ \lfloor n/d \rfloor\le u_i\le \lceil n/d \rceil$, write $C(n,S;L)$ (or
$C(n,e;L)$ when $S=\{e\}$).
For a fixed ordered pattern $\boldsymbol a=(a_1,\ldots,a_d)\in\mathcal P$,
denote
\[
 C_{\boldsymbol a}(u_1,\ldots,u_d)
 :=C(u_1,\ldots,u_d,e;L),
 \qquad \rho_U(e)=\boldsymbol a.
\]
A pattern $\boldsymbol a$ is called \emph{simple} if it can be obtained by a permutation of  $(2,1,\dots,1,0,\dots,0)$.

\begin{thm}\label{thm-c}
		Let  $d>k\ge3$ be integers, and let	 \( L \in \mathcal{F}(k,d) \)  with $l$ vertices. 
	For  every  sufficiently large $n$ and every ordered pattern $\boldsymbol a=(a_1,\ldots,a_d)\in\mathcal P$, there exists a constant \( \alpha_{L} > 0 \) with 		 
		\[
C_{\boldsymbol a}(n,L)=\alpha_L\sum_i\binom{a_i}{2}n^{l-k}+O_L(n^{l-k-1}).
\]
	Thus $c(n,L)=C_{\boldsymbol a}(n,L)=\alpha_Ln^{l-k}+O_L(n^{l-k-1})$ with a simple pattern $\boldsymbol a$.
	\end{thm}
	\begin{proof}
		Let $L=\mathcal{K}^{(k)}(F)$, where \( F \) is  color-critical graph  with $f$ vertices. Let $l:=|V(L)|$.  We say $V(F)$  the \emph{original vertex set} of $L$.   
		Let the \( d \) parts of \( T^{(k)}(n,d) \) be \( V_1, \dots, V_d \) with \( |V_1| \ge \cdots \ge |V_d|  \), and define \( n_i := |V_i| \) for \( i \in [d] \). Let \( H = T^{(k)}(n,d)\cup e \).
		
		By the definition of a color-critical graph, there exists at least one critical edge \( uv \) of $F$, i.e., the chromatic number of \( F - uv \) is \( d \). Let $L_{uv}$ be the $k$-edge obtained by expanding $uv$.
        Let \( \kappa \) denote the set of all critical edges of \( F \). Take any \( uv \in \kappa \), and let \( C_{uv} \) denote all partitions of \(V(F)\setminus \{u,v\}\) into \( d \) disjoint independent sets \( \{U_1, \dots, U_{d}\} \) such that \( U_1\cup\cdots\cup U_{d}=V(F)\setminus \{u,v\} \).
		
		Take any \( \{U_1, \dots, U_{d}\} \in C_{uv} \), and   define a binary function \( f: [d] \times [d] \to \{0,1\dots,e(F)\} \) such that for any \( \{i,j\} \in [d] \times [d] \), \( f(i,j) = e(F[U_i, U_j]|) \), which is the number of edges in \( F \) that intersect both \( U_i \) and \( U_j \). Thus, \( f(i,i) = 0 \) for any $i\in[d]$ and \[\sum_{1 \le i < j \le d} f(i,j) = e(F) - d_F(u) - d_F(v)+1.\] 
		For vertex $u$,	define a function \( f_u: [d] \to \{0,1\dots,e(F)\} \)  such that for any \( j \in [d] \), \( f_u(j) = e(F[u, U_j]) \), i.e., the number of edges in \( F \) between \( u \) and \( U_j \). Similarly for $v$, define a function \( f_v: [d] \to \{0,1\dots,e(F)\} \)  such that for any \( j \in [d] \), \( f_v(j) = e(F[v, U_j])\). Thus, 
		\[ \sum_{j \in [d]} f_u(j) = d_F(u) -1, \quad 
		\sum_{j \in [d]} f_v(j) = d_F(v)-1. \]
		
		Comparing \( H \), \( L \) and \( F \), a copy of \( L \) in \( H \) is  an edge-preserving mapping of \( L \).
Consider an injective embedding of $L$ into
$T^{(k)}(n,d)\cup\{e\}$.
Color each original vertex by the part containing its image.
Since $\chi(F)=d+1$, some original edge has endpoints of the same color. Its expanded edge is therefore mapped to $e$.
Every other expanded edge is crossing.
Thus there is exactly one original edge $uv$ whose endpoints have
the same color.
It follows that $uv\in\kappa$.
The images of $u$ and $v$ lie in one part, 
say \( V_i \) without loss of generality. 
There are $2{a_i\choose 2}$ ways to map the labelled vertices $u,v$
to distinct vertices of $e\cap V_i$.
There are $(k-2)!$ ways to map the other vertices of $L_{uv}$
to the remaining vertices of $e$.     
        We divide the mappings of the remaining edges and vertices into three types:
		(1) The vertices of \( F \). For   \( \{U_1, \dots, U_{d}\} \in C_{uv} \), the number of choices for \( U_1, \dots, U_{d} \) in \( H \) is 
		\[
		(n_1 - a_1)_{|U_1|} \cdots (n_d - a_d)_{|U_d|}.
		\]
		(2) The vertices corresponding to the extended edges incident to \( u \) (or \( v \)). For any $ j\ne i$, take an edge \( uu' \) with \( u' \in U_j \), then  \( uu'\) remains an edge in \( L \) after extension.
		Since  each edge in \( H - e \) intersects at most one element of \( V_r \) for each $r\in[d]$,   the \( k-2 \) vertices extended from \( uu' \) are mapped into \( k-2 \) vertex sets of \(\{ V_1, \dots, V_d\}\setminus\{V_i, V_j\} \). The number of choices for these \( k-2 \) vertices is at least
		\[
		\sum_{\{i_1, \dots, i_{k-2}\} \in \binom{[d] \setminus \{i,j\}}{k-2}} (n_{i_1}-f) \cdots( n_{i_{k-2}}-f).
		\]
		Note that the number of choices for such edge $uu'$  is $f_u(j)$ (or $f_v(j)$).
		(3) The vertices corresponding to the extended edges between \( U_{i'} \) and \( U_{j'} \). For any $1\le i'<j'\le d$, take an edge \( u_1u_2 \) with \( u_1 \in U_{i'} \) and \( u_2 \in U_{j'} \). Also, \( u_1u_2 \) remains an edge in \( L \) after extension. Thus, the \( k-2 \) vertices extended from \( u_1u_2 \) are mapped into \( k-2 \) vertex sets of \(\{ V_1, \dots, V_d\}\setminus\{V_{i'}, V_{j'}\} \) in \( H \). The number of choices for these \( k-2 \) vertices is at least
		\[
		\sum_{\{i_1, \dots, i_{k-2}\} \in \binom{[d] \setminus \{i',j'\}}{k-2}} (n_{i_1}-l) \cdots( n_{i_{k-2}}-l).
		\]Note that the number of choices for such edge $u_1u_2$  is $f(i'j')$.
		
		Let \( |\text{Aut}(F) |\) denote the number of automorphisms of \( F \), and define
		\[
		\varGamma (i) =\prod_{j\in[d]} \left( \sum_{\{i_1, \dots, i_{k-2}\} \in \binom{[d] \setminus \{i,j\}}{k-2}} n_{i_1} \cdots n_{i_{k-2}} \right) ^{f_u(j) + f_v(j)} 
		,
		\]
		and 
		\[\varGamma=\prod_{1 \le i' < j' \le d} \left( \sum_{\{i_1, \dots, i_{k-2}\} \in \binom{[d] \setminus \{i',j'\}}{k-2}} n_{i_1}\cdots n_{i_{k-2}}\right) ^{f(i',j')}.\]
		Then summing over all $uv \in \kappa$ and all $C_{uv}$ yields 
		\[
		\N_H(L) \ge (1-o(1))\left\lbrace \frac{(k-2)!}{|\text{Aut}(F)|} \sum_{i \in [d]} 2\binom{a_i}{2} \sum_{uv \in \kappa} \sum_{C_{uv}}  \left( \prod_{r \in [d]} (n_r - a_r)_{|U_r|} \right) \varGamma (i)\varGamma \right\rbrace ,
		\]
		and \[
		\N_H(L) \le\left\lbrace \frac{(k-2)!}{|\text{Aut}(F)|} \sum_{i \in [d]} 2\binom{a_i}{2} \sum_{uv \in \kappa} \sum_{C_{uv}}  \left( \prod_{r \in [d]} (n_r - a_r)_{|U_r|} \right) \varGamma (i)\varGamma \right\rbrace .
		\]	
		The equality $
		\min\sum_{i \in [d]} \binom{a_i}{2}
		$
		is achieved when \(a_i = 2\) for the part \(n_i\) and \(a_j = 0\) or \(1\) for all other \(j\).
		Under this choice,  the value of
		\[
		\sum_{i \in [d]} \binom{a_i}{2} \sum_{uv \in \kappa} \sum_{C_{uv}}
		\left( \prod_{r \in [d]} (n_r - a_r)_{|U_r|} \right) \varGamma(i) \varGamma
		\]
		is minimized. Therefore,
		\[
		c(n,L) = \Theta\!\left(n^{f - k + (k-2)\left( \sum_{j \in [d]} (f_u(j) + f_v(j))
			+ \sum_{1 \le i < j \le d} f(i,j) \right)}\right) = \Theta(n^{l - k}).
		\]	
		This completes the proof.	
	\end{proof}

	\begin{lem}\label{lem-cs}
		Let integers $l>k\ge3$, and   $L\in \mathcal{F}(k,d)$ with $l$ vertices. 
           For  every  sufficiently large $n$ and every   pattern $\boldsymbol a=(a_1,\ldots,a_d)\in\mathcal P$, there are constants $\zeta_{L}, \zeta'_L$ such that the following holds. Suppose  $C_{\boldsymbol a}(n,L)=C(n_1',\cdots,n_d';L)$ and  $n_1+\cdots+ n_d=n$. Let $A_i=n_i-n_i'$ and $A=\max_{i\in[d]}\{|n_i-n_i'|\}$,
		then 
		\[|C_{\boldsymbol a}(n_1,\cdots,n_d;L)-C_{\boldsymbol a}(n,L)-\sum_{i\in[d]}{a_i\choose 2}\zeta_{L}A_1n^{l-k-1}|\le \zeta'_LA^2n^{l-k-2} .\]	
	\end{lem}
	\begin{proof}
		By Theorem \ref{thm-c}, we see that $C_{\boldsymbol a}(n,L)$ is a polynomial in $n$ of degree $l-k$. 	
Let $R_1$ be the homogeneous
part of degree $l-k$ of the polynomial for the simple pattern with two vertices in part $1$, one
vertex in each of parts $2,\dots,k-1$, and no vertices in the other parts. Then   
		by Theorem \ref{thm-c}, we observe that 
\[R_1= \frac{2(k-2)!}{|\text{Aut}(F)|}\sum_{uv \in \kappa} \sum_{C_{uv}}
		 \prod_{r \in [d]} n_r^{|U_r|}  \varGamma(1) \varGamma
		, \quad C_{\boldsymbol a}(n,L)=\sum_{i\in[d]}{a_i\choose 2}R_1+O_L(n^{l-k-1}).\]
And \(R_1\) is symmetric with respect to the variables \(n_2,\dots,n_{d}\). At $(1/d,\ldots,1/d)$, this imply that 
$\frac{\partial R_1}{\partial n_j}\left(1/d,\ldots,1/d\right) =\frac{\partial R_1}{\partial n_i}\left(1/d,\ldots,1/d\right)$ for each $j,i\neq 1$.
Let \begin{align}\zeta_L=\left(\frac{\partial R_1}{\partial n_1} -\frac{\partial R_1}{\partial n_i}\right)\left(\frac1d,\ldots,\frac1d\right).\end{align}

		If $A=0$ then the result is trivial. Now we assume  $A\ge 1$.  Taking the first-order Taylor expansion about $(n_1',\cdots,n_d')$, there is \begin{align}\label{eq-cnl}	C_{\boldsymbol a}(n_1'+A_1,\ldots,n_d'+A_d;L)-&C_{\boldsymbol a}(n_1',\ldots,n_d';L)\notag\\
		= &\sum_{i\in[d]}{a_i\choose 2}\sum_{j\in[d]} A_j \frac{\partial C_{\boldsymbol a}}{\partial n_i}(n_1',\ldots,n_d')
		+ O(A^2 n^{l-k-2}).\end{align}	
		In addition, since  $\frac{\partial C_{\boldsymbol a}}{\partial n_i}(n_1',\ldots,n_d')$ is a polynomial in $n$ of degree $l-k-1$ and    $|n_i'-n/d|\le 1$ for all $i\in[d]$.  By the mean value theorem, for each $i\in[d]$, we have \[\left| \frac{\partial C_{\boldsymbol a}}{\partial n_i}(n_1',\ldots,n_d')
		-
		\frac{\partial C_{\boldsymbol a}}{\partial n_i}\!\left(\frac{n}{d},\ldots,\frac{n}{d}\right)\right| 
		= O(A n^{l-k-2}).\]
Since $\sum_{j=1}^d A_j=0$, we obtain
		\begin{align*}\sum_{i\in[d]} A_i\frac{\partial C_{\boldsymbol a}}{\partial n_i}\!\left(\frac{n}{d},\ldots,\frac{n}{d}\right)=&A_1 
		\frac{\partial C_{\boldsymbol a}}{\partial n_1}\left(\frac{n}{d},\ldots,\frac{n}{d}\right)
		+
		\sum_{i\ge 2}A_i\frac{\partial C_{\boldsymbol a}}{\partial n_2}\left(\frac{n}{d},\ldots,\frac{n}{d}\right)\\
		=&A_1 
		\frac{\partial R_1}{\partial n_1}\left(\frac{n}{d},\ldots,\frac{n}{d}\right)
		+
		\sum_{i\ge 2}A_i\frac{\partial R_1}{\partial n_2}\left(\frac{n}{d},\ldots,\frac{n}{d}\right)+O_L(n^{l-k-2})\\
        =&A_1\left(\frac{\partial R_1}{\partial n_1}-
		\frac{\partial R_1}{\partial n_2}
		\right)\!\left(\frac{n}{d},\ldots,\frac{n}{d}\right)+O_L(n^{l-k-2})
		.\end{align*}
Finally, replace the sum $ \sum_{i\in[d]}A_i\frac{\partial C_{\boldsymbol a}}{\partial n_i}(n_1',\ldots,n_d')$
		by above  equality in \eqref{eq-cnl}   yields the desired lemma.
	\end{proof}

	By a   similar proof to  Theorem \ref{thm-c} and Lemma \ref{lem-cs} , we establish the following  Theorems.
	
	\begin{thm}\label{thm-sl}
		Let  $d>k\ge3$ be integers, and let  $L \in \mathcal{F}(k,d)$  with $l$ vertices. For a non-crossing edge set $S$, if $c(n, S; L) \neq 0,$ then for all   sufficiently large $n$, 
		$
		c(n, S; L) =\Theta(n^{l-|V(S)|})
		$.	
	\end{thm}

	\begin{lem}\label{lem-Cs}
		Let  $l>k\ge3$  be integers, and let	$L\in \mathcal{F}(k,d)$.  Suppose  $C(n,S;L)=C(n_1',\cdots,n_d';S;L)$. Let    $s=\max_{i\in[d]}\{|n_i-n_i'|\}$ with $n_1+\cdots+ n_d=n$,
		then for  all sufficiently large $n$,
		\[|C(n_1,\cdots,n_d;S;L)-C(n,S;L)|=O_{L}(sn^{l-|V(S)|-1}).\]	
	\end{lem}
\subsection{Stability and structural
estimates}\label{s2.2}
We first state the removal lemma  and proof the stability lemma. These  give a partition with few
bad and missing edges.
We then give a local  estimate.
\begin{thm}[Hypergraph Removal Lemma \cite{12,23,27,30}]\label{thm-remove}
		Let $k \geq 2$ be fixed, and let $L$ be a $k$-uniform hypergraph on $l$ vertices.
		For every $\varepsilon > 0$, there exist $\delta> 0$ and $n_0 = n_0(\varepsilon, L)$ such that the following holds for all $n > n_0$.
		If an $n$-vertex $k$-uniform hypergraph $H$ contains at most $\delta n^l$ copies of $L$, then there exists at most $ \varepsilon n^k$ edges such that the graph obtained from $H$ by deleting those edges is $F$-free.
	\end{thm}
	Two $k$-graphs of the same order are \emph{$m$-close} if adding
or deleting at most $m$ edges makes them isomorphic.

	 \begin{thm}[ Pikhurko~\cite{Pikhurko} ]\label{thm-clique-stable}
		Let  \(d> k \ge 3\)  be integers. 
		For every \(\varepsilon > 0\), there exist \(\delta = \delta(k,d,\varepsilon) > 0\)  such that the following holds for all sufficiently large \(n\): if an \(n\)-vertex  \(k\)-graph \(G\) is \(\mathcal K^{(k)}(K_{d+1})\)-free and satisfies
		$e(G) \ge t^{(k)}(n,d)- \delta  \binom{n}{k}$,
		then \(G\) is \(\varepsilon \binom{n}{k}\)-close to the \(d\)-partite  $k$-graph \(T^{(k)}(n,d)\).
	\end{thm}	
   \begin{thm}[Stability Theorem]\label{thm-stable}
	Let  $d > k \ge 3$ be integers, and let $L\in \mathcal{F}(k,d)$.  For every $\varepsilon > 0$, there exist $\delta = \delta(k,d,\varepsilon) > 0$  such that the following holds for all sufficiently large \(n\): if an $n$-vertex $L$-free  $k$-graph  $G$ satisfies $e(G)\ge t^{(k)}(n,d) - \delta  \binom{n}{k}$, then $G$ is $\varepsilon \binom{n}{k}$-close to the  $T^{(k)}(n,d)$. \end{thm}
	\begin{proof}
	Assume that $L= \mathcal{K}^{(k)}(F)$, we have  $e(L)=e(F), |V(L)|=|V(F)|+e(F)(k-2)$ and  $\chi(F)=d+1$.
	Let $\varepsilon>0$ be given. Choose $\delta>0$ sufficiently small such that the
	stability theorem for $\mathcal{K}^{(k)}(K_{d+1})$ holds with respect to $\varepsilon/2$.
	Assume further that $\delta\le \varepsilon$ and  $n$ be sufficiently large.  Let 
	$G$ be an $L$-free $k$-graph on vertex set $[n]$ satisfying
	$e(G)\ge t^{(k)}(n,d)- {\delta\over 2} \binom{n}{k}$.
	We call a pair $\{x,y\}\subseteq V(G)$ \emph{sparse} if it is contained in at most
	\[
	m=\left(|V(F)|+(k-2)e(F)\right)\binom{n}{k-3}
	\]
	edges of $G$. Let $G'$ be the subgraph obtained from $G$ by removing all edges that contain at least one sparse pair. Since each sparse pair is contained in at most $m$ edges, and there
	are at most $\binom{n}{2}$ pairs of vertices, the total number of removed edges is at
	most
	$\binom{n}{2}\cdot m < \frac{\delta}{2}\binom{n}{k}$
	provided $n$ is sufficiently large.
	Hence,\[e(G')\ge e(G)-\frac{\delta}{2}\binom{n}{k}\ge t^{(k)}(n,d)-\delta\binom{n}{k}.\]
	
	We now show that $G'$ is $\mathcal{K}^{(k)}(K_{d+1})$-free. Suppose otherwise that $G'$ contains a copy
	of $\mathcal{K}^{(k)}(K_{d+1})$ on vertex set $S\subseteq G'$. By definition of $G'$, every pair $\{x,y\}\subseteq S$
	is not sparse in $G$, meaning that each such pair is contained in more than $m$ edges of $G$.
	This abundance of edges allows us to extend the partial embedding of $L$ inside
	$G$ inductively: given a core embedding on $S$, one can always find an additional edge
	containing a prescribed vertex while avoiding previously used vertices, because each pair
	has sufficiently many available extensions. Consequently, $G$  contain a copy of
	$L$, contradicting the assumption.
	Therefore $G'$ is $\mathcal{K}^{(k)}(K_{d+1})$-free.
	
	By the stability Theorem \ref{thm-clique-stable} for $\mathcal{K}^{(k)}(K_{d+1})$, the graph $G'$ is $\varepsilon/2$-close
	to  $T^{(k)}(n,d)$. That is, by modifying at most
	$\varepsilon/2\binom{n}{k}$ edges, $G'$ can be transformed into $T^{(k)}(n,d)$.
	Finally, the triangle inequality implies that $G$ itself is
	$\varepsilon$-close to $T^{(k)}(n,d)$, since
$	|E(G)\triangle E(G')|\le \frac{\delta}{2}\binom{n}{k}$ and	$\delta\le \varepsilon$.
	This completes the proof.
	\end{proof}

	\begin{remark} \label{remark-mb}Fix a $k$-graph $L\in \mathcal{F}(k,d)$ with $l$ vertices. Let constants satisfying  $1>\delta_4\gg\cdots\gg \delta_1\gg1/n$, and let $H$ be an $n$-vertex $k$-graph with $e(H)\ge t^{(k)}(n,d)$.
	Suppose that $\N_L(H)\le\delta_1 n^{l}.$
	By Theorem \ref{thm-remove}, there exists a set of at most ${1\over2}\delta_2 n^k$ edges whose removal eliminates all copies of $L$, in other words,  there exists an $L$-free $k$-graph $ H'$ satisfying
	$|E( H')|\ge|E (H)|-{1\over2}\delta_2n^k\ge t^k(n,d)-{1\over2}\delta_2n^k.$
	Next, applying Theorem \ref{thm-stable}, there exists an $d$-partition
	$
	V_1\cup\cdots\cup V_d
	$
	of $V( H)$ such that at least $e(H')-{1\over2}\delta_3n^k$ edges of $H'$ are crossing edges.
    Let $h_i$ denotes the number of edges of $H$ intersecting exactly $i$ of the parts, and then $\sum_{i\in[ k-1]}ih_i\le \delta_3n^k$, $h_k\ge e(H')- \delta_3n^k $. Among all such partitions, choose one maximizing	$\Phi:=\sum_{i\in[ k]}ih_i$, which is  possible since 
	$\Phi\ge ke(H)-(k-1)\sum_{i\in[ k-1]}h_i\ge k\cdot t^{(k)}(n,d)-(k-1){1\over2}\delta_3n^k$ and $\Phi\le ke(H')-\sum_{i\in[ k-1]}h_i\le k\cdot t_d^{(k)}(n)-\delta_3n^k$.
	Denote by $K^{(k)} (V_1,\dots, V_d)$ the complete $d$-partite $k$-graph with vertex partition $
	V_1\cup\cdots\cup V_d
	$.
	By Lemma \ref{lem-st}, for each $i\in[d]$, $n/d- \delta_4n\le n_i:=|V_i|\le n/d+ \delta_4n$. 
	\end{remark}
	Let $B(H)=E(H)\setminus E(K^{(k)} (V_1,\dots, V_d) )$
	be the collection of \emph{bad edges}, and let
	$M(H)=E(K^{(k)} (V_1,\dots, V_d) )\setminus E(H)$ 	be the collection of \emph{missing edges}. Therefore 
	\begin{align}\label{eq-mb}
	|M(H)|+(e(H)-t^{(k)}(n,d))\le|B(H)|&=|E (H)\setminus E(K^{(k)} (V_1,\dots, V_d) )|\notag\\&\le {1\over2}\delta_2n^k+{1\over2}\delta_3n^k\le \delta_3n^k.
	\end{align}

A \emph{potential copy} of $L$ is a copy from $K^{(k)} (V_1,\dots, V_d) +B(H)$, and we say a potential copy of $L$ is \emph{missing copy}   if it  contains some  edge from $M(H)$. Write $d_M(x)$ for the number of missing edges containing $x$.
Next, we establish a stronger upper bound on $B(H)$. 
	This structural estimate will be the key ingredient in the proof of the main theorem.	
	\begin{lem} \label{lem-b}\label{lem-local-cleaning} Let  $d > k \ge 3$ be integers. Fix $L\in \mathcal{F}(k,d)$ with $l$ vertices. Let constants satisfying  $1>\delta_7\gg\cdots\gg \delta_1\gg1/n$, and let  $H$, $\Phi$, $K=K^{(k)}(V_1,\ldots,V_d)$, $B(H)$, and
$M(H)$ satisfying Remark \ref{remark-mb}.
Suppose that $\N_L(H)\le\delta_1n^{l-1}$, then there exist constant $C$ such that
$|B|\le C\cdot \mathcal N_L(H)n^{k-l}$.
\end{lem}

	\begin{proof} 
    
		Let $B=B(H), M=M(H), L=\mathcal{K}^{(k)}(F)$.	Let $L(e)$ denote the number of copies of $L$ in $H$ that use the  edge $e$ as the unique edge from $B$. Let $D:=\{e\in  B\mid L(e)< (1-2\delta_5)c(n,L)\}$.

		\begin{clm}\label{clm-dmx}
	  For every  $e\in D$.   Suppose that   $e=x_1x_2\cdots x_k$ and $x_1,x_2$  in a common
part of $K$. 		There exists a vertex
			$x\in \{x_1,x_2\}$
			such that
			$
			d_M(x)\ge\delta_8n^{k-1}.
			$	
		\end{clm}
		
		Since $F$ is color-critical with $\chi(F)=d+1$, there exists an edge
		$uv\in E(F)$ such that $F-uv$ is $d$-colorable. Fix a proper $d$-coloring of
		$F-uv$, and let $u_1,\ldots,u_d$ denote the sizes of the corresponding color classes.
		Choose disjoint vertex sets
		$U_i\subseteq V_i\setminus\{x_1,x_2,\dots ,x_k\}$
		with $|U_i|=u_i$ for all $i\in[d]$. 
		For every pair of distinct vertices $
		\{a,b\} 
		$ with
		$a\in \bigcup_{i=1}^{d} U_i$ and 
		$b\in \{x_1,x_2\}\cup\bigcup_{i=1}^{d} U_i
		$, 
		select  vertex set
		$
		V_{ab}
		$(might be empty)
		from $k-2$ part different from those containing $a$ and $b$, and  does not contain any vertex chosen   previously. 
		The resulting collection of sets determines a potential copy of $L$ containing the edge $e$,  with  edge set $\{e\}\cup \{\{a,b\}\cup V_{a,b}\}_{a,b}$  and original vertex set $\{x_1,x_2\}\cup\bigcup_{i=1}^{d} U_i$.
		Recall the proof of $c(n,L)$, note that 
		the number of such potential copy of $L$ is at least
		$
		(1-\delta_5)c(n,L)
		$ since the fact that 
		$||V_i|-\frac{n}{d}|\le\delta_4n
		$ for every $i\in[d]$.

		At least
		$
		\delta_5c(n,L)
		$
		of these potential copies is missing, since less than
		$
		(1-2\delta_5)c(n,L)
		$
		of these potential copies are present in $ H$.
		 On the other hand, if at least
		$
		\frac{\delta_5}{2}c(n,L)
		$
		of these  missing copies of $L$ contain missing edges from $M$ omitting edge $e$. Since each missing edges from $M$ is counted at most
		$
		n^{l-2k}
		$
		times, we have
		$
		\frac{\delta_5}{2}c(n,L)/n^{l-2k}
		\le |M|
		<\delta_3n^k,
		$
		a contradiction.
		This implies that at least
		$
		\frac{\delta_5}{2}c(n,L)
		$
		of these missing copies of $L$ contain a missing edge from $M$  intersecting edge $e$. In fact, every such missing edge shares exactly one vertex $(x_1\ or\ x_2)$ with $e$, since any two edges of $L$ intersect in at most one vertex and $x_1,  x_2$ is the original vertices. 
		On the other hand, every such missing edge is counted at most
		$
		n^{l-2k+1}
		$
		times, then there exists a constant $\delta_8\gg\delta_5$ and one vertex $x\in\{x_1,x_2\}$ such that
		$
		d_M(x)\ge\frac{\delta_5}{2}c(n,L)/(2n^{l-2k+1})\ge\delta_8n^{k-1}.
		$ 
		By symmetry, suppose the inequality holds for $x_1$. This completes the proof.

	Fix a vertex $x\in V_j$, let
		$
		Y_j(x)
		$
		denote the set of edges in $B$ containing $x$ and at least two vertices from $V_j$. Next  we  show  that
		$
		|Y_j(x)|\le 6\delta_7n^{k-1}.
		$
		Let $d_{H}(uv)$ denote  the numbeb of edges contains vertices $u$ and $v$. For each $i\in[d]$,  define
		$
		Y_i=\{v\in V_i:d_{\mathcal H}(xv)\ge \delta_7n^{k-2}\}.
		$
		
		\begin{clm}
			There are at most  $d-1$ indices  $i\in[d]$ such that
			$
			|Y_i|\ge\delta_7n.
			$
		\end{clm}

		Since $F$ is  color-critical, there exists a vertex $u$ satisfying  $\chi(F-u)\ge d$. 	 Let $\{U'_1,\dots, U_d'\}$  be a proper $d$-colorings of $F-u$.
		Suppose for every $i\in[d]$,
		$
		|Y_i|\ge \delta_7n
		$. Then the number of choices of
		$
		U_1\times\cdots\times U_d\subseteq  Y_1\times\cdots\times Y_d
		$ with $|U_i'|=|U_i|$ for all $i\in[d]$
		is at least
		$
		\prod_{i=1}^{d}{\delta_7n-f\choose  |U_i|}> \frac{1}{2}(\delta_7n/f)^{f-1}.
		$
		For any   two vertices
		$a,b$ from  distinct parts of $ U_1\cup\cdots\cup U_d
		$,
		select  a $(k-2)$-tuple $
		V_{ab}
		$ of  vertices 	from $k-2$ parts different from those containing $a$ and $b$, such that $x\notin W_{ab}$. Then there are at least
		$
		{d-2\choose k-2}(n/d-\delta_4)n^{k-2}>\frac{1}{2}(n/d)^{k-2}
		$
		choices for each $V_{ab}$.
		For every vertex
		$
		c\in \bigcup_{i=1}^{d}U_i,
		$
		let $V_c$ be a set of $k-2$ vertices such that
		$
		\{x,c\}\cup V_c
		$ is an edge of $H$.
		Following the definition of $Y_i$, there are at least
		$
		\delta_7n^{k-2}
		$
		choices for each such $V_c$.
		
		Consequently, these 	choices of the above $
		(l-1)
		$ additional vertices form a potential copy of $L$ with original vertices
		$
		x\cup (\bigcup_{i=1}^{d}U_i)
		$, 
		and the total number of  such copies is at least
		\[\frac{1}{2}(\delta_7n/f)^{f-1}\cdot\left(\frac{1}{2}(n/d)^{k-2}\right)^{e(F)-d_{F}(x)}(\delta_7n^{k-2})^{d_{F}(x)}
		>(\delta_6n)^{l-1}, 
		\]
		where	$\delta_6>0$ be a sufficiently small  constant.
		If at least
		$
		\frac{\delta_6}{2}n^{l-1}
		$
		of these potential copies lying in $H$, then
		$
		\N_H(L)>
		\delta_1 n^{l-1},
		$
		a contradiction.
		Thus at least
		$
		\frac{\delta_6}{2}n^{l-1}
		$
		of these potential copies are missing copies (i.e., contain an edge  from $M$). Since $x$ is fixed and each missing edge  is counted at most
		$
		n^{l-k-1}
		$
		times. It implies that $\frac{\delta_6}{2}n^{l-1}/n^{l-k-1}\le|M|<\delta_3n^{k}$,
		again contradicting.
		This proves the Claim.

		Suppose otherwise $
		|Y_j(x)|>6\delta_7n^{k-1}
		$. Let $T(x)$ be the family of $(k-1)$-sets $S$ such that
$S\cup\{x\}\in Y_j(x)$.
		Recall the definition of $Y_j(x)$, we have 
		$
		|T(x)|=|Y_j(x)|
		$, and for any $K\in T(x)$, $K\cap V_j\neq \emptyset$. 
		If $|Y_j|<\delta_7n$ then 		\begin{align*}|Y_j(x)|
		\le
		\sum_{v\in V_j}d_{Y_j(x)}(v)
		=
		\sum_{v\in Y_j}d_{Y_j(x)}(v)
		+
		\sum_{v\in V_j\setminus Y_j}d_{Y_j(x)}(v)
		<|Y_j|n^{k-2}+|V_j\setminus Y_j|\delta_7 n^{k-2}
		\le
		2\delta_7n^{k-1},
		\end{align*}
	a	contradiction.
		Hence $
		|Y_j|\le \delta_7n$ and one of the sets
		$
		Y_m, m\neq j
		$
		has size smaller than $\delta_7n$ by the Claim above. By symmetry, assume that
		$
		|Y_d|<\delta_7n.
		$
		The number of edges of $ H$ containing $x$ and intersecting $V_d$ is at most
		\[\sum_{v\in V_d}d_{H}(xv)
		=
		\sum_{v\in Y_d}d_{H}(xv)
		+
		\sum_{v\in V_1\setminus Y_1}d_{H}(xv)
		<|Y_d|n^{k-2}+|V_d\setminus Y_d|\delta_7 n^{k-2}
		\le
		2\delta_7n^{k-1}.
		\]
		On the other hand, the number of edges in $Y(x)$ avoiding $V_d$ is at least
		$
		|Y_j(x)|-2\delta_7n^{k-1}
		>
		4\delta_7n^{k-1}.
		$
		
		Finally, move the vertex $x$ from $V_j$ to $V_d$. Every edge in $Y_j(x)$ avoiding $V_d$ increases its contribution to
		$
		\Phi
		$
		by exactly one. Every edges of $H$ containing $x$ and intersecting $V_d$ decreases its contribution to
		$
		\Phi
		$ by at most one. 
		All remaining edges containing $x$, namely those whose other $k-1$ vertices lie in
		$
		\bigcup_{i\neq j} V_{i},
		$
		do not change their contribution. Hence the total change in
		$
		\Phi
		$
		is at least
		$
		4\delta_7n^{k-1}-2\delta_7n^{k-1}
		=
		2\delta_7n^{k-1}>0.
		$
		This contradicts the choice of the partition as one maximizing
		$
		\Phi
		$.
	It shows that the assumed failure is impossible. Hence $|Y_j(x)|\le 6\delta_7n^{k-1}$.

        Let $e\in D, f\in M$, we say $(e,f)$ is a good pair if there  exist a vertex $u\in e\cap f$ such that $u\in e\cap V_i$   with $|e\cap V_i|\ge 2$.
	Since $|Y_j(x)|\le6\delta_7n^{k-1}$,  for each missing edge  $f\in M$,  the  number of good pairs $(e,f)$ with $e\in D$ is at most $6k\delta_7n^{k-1}$. 
By Claim \ref{clm-dmx}, for each bad edge $e\in D$,  the  number of good cuples $(e,f')$ with $f\in M$ is at least $\delta_8n^{k-1}$. 
Since $|M|\le |B|$, a double counting gives $|D|\delta_8n^{k-1}\le 6k\delta_7n^{k-1}$. Choose the constants ao that $1/2\delta_8\gg 6k\delta_7$.
Thus there exist constant $C $ such that 
$\N_L(H)\ge  (|B|-|D|)(1-2\delta_5)c(n,L)\ge |B|Cn^{k-l}$.
	\end{proof}

\subsection{Balancing and construction}\label{s2.3}
we first understand the structure of
	$d$-partite $k$-graphs whose number of edges is close to
	$t^{(k)}(n,d)$. Next, we bound the change in the
copy count when balance the parts.
Finally, we give a random edge addition estimate for copies
involving several added edges.

	\begin{lem}\label{lem-st}
		Let \(d> k\ge3\)be integers, and let $t=t(n)\ge0$ satisfy $t=o(n^k)$.
Suppose that $x_1,\ldots,x_d\ge0$ are integers with $\sum_{i=1}^d x_i=n$ and 
\[
 S_k=\sum_{1\le i_1<\cdots<i_k\le d}x_{i_1}\cdots x_{i_k}
 \ge t^{(k)}(n,d)-t.
\]
Then for 
sufficiently large $n$,   $\max_{i\in[d]}|x_i-n/d|=O_{k,d}(1+\sqrt {t/n^{k-2}})$. 
In particular, there are balanced integers
$n_1,\ldots,n_d$ with $\sum_{i=1}^d n_i=n$ such that
		$\max_{i\in[d]}|x_i-n_i|=O_{k,d}(\sqrt{t/n^{k-2}})$.
	\end{lem}
	
	\begin{proof}
		Define \(\mu := n/d\),   \(y_i := x_i - \mu\), $s:=\max_{i\in[d]}|y_i|$. Then \(\sum_{i=1}^r y_i = 0\) and  
		\[
		S_k =\sum_{1\le i_1<\cdots<i_k\le d} (y_{i_1}+\mu)\cdots (y_{i_k}+\mu) =\sum_{j=0}^k \mu^{k-j} \binom{r-j}{k-j} e_j(y),
		\]
		where \(e_j(y)=\sum_{1\le i_1<\cdots<i_j\le d} y_{i_1}\cdots y_{i_j}\) is the \(j\)-th elementary symmetric sum of \(y_1,\dots,y_d\) (with \(e_0=1\)).  
		Since \(e_1(y)=0\), we have
		\[
		S_k = \binom{r}{k}\mu^k + \mu^{k-2}\binom{r-2}{k-2}e_2(y) + \sum_{j=3}^k \mu^{k-j}\binom{r-j}{k-j}e_j(y).
		\]
		Because \(e_2(y) = \frac12\bigl((\sum_{i\in[d]} y_i)^2 - \sum_{i\in[d]} y_i^2\bigr) = -\frac12\sum_{i\in[d]} y_i^2\), we obtain
		\begin{align}
		S_k = \binom{d}{k}\mu^k - \frac12\mu^{k-2}\binom{d-2}{k-2}\sum_{i=1}^d y_i^2 + \sum_{j=3}^k \mu^{k-j}\binom{d-j}{k-j}e_j(y). \label{eq1}
		\end{align}

		Note that \(s=o(n)\) for every \(i\in[d]\), since $
		S_k > t^{(k)}(n,d) - t
		$ and $t=o(n^{k})$. Then \(\sum_{i\in[d]} y_i^2 \le d s^2\) and \(|e_j(y)| \le \binom{d}{j}s^j\).  Thus
		\[
		\left|\sum_{j=3}^k \mu^{k-j}\binom{d-j}{k-j}e_j(y)\right| \le \sum_{j=3}^k \mu^{k-j}\binom{d-j}{k-j}\binom{d}{j}s^j = O(\mu^{k-3}s^3).
		\]
		From (\refeq{eq1}) we obtain $S_k \ge \binom{d}{k}\mu^k - \frac12\mu^{k-2}\binom{d-2}{k-2}\cdot d s^2 - O(\mu^{k-3}s^3).$
	So  	\begin{align}\label{eq4}
		\binom{d}{k}\mu^k-t^{(k)}(n,d) \le \frac12\mu^{k-2}\binom{d-2}{k-2}\cdot d +O(\mu^{k-3}). 
		\end{align}

		On the other hand, assume without loss of generality that \(|y_1|=s\).
		Since \(\sum y_i = 0\), the minimum of \(\sum y_i^2\) subject to \(|y_1| =s\) is attained when
		$
		y_1 = s, y_2 = \dots = y_d = -\frac{s}{d-1}$,
		which implies
		\begin{align}
		\sum y_i^2 \ge \frac{d}{d-1}{s'}^2. \label{eq2}
		\end{align}	
		For \(j\ge 3\), \(|e_j(y)| \le \binom{d}{j} (\max_{i\in[d]}|y_i|)^j = O({s}^j)\).  Hence
		$
		\left|\sum_{j=3}^k \mu^{k-j}\binom{d-j}{k-j}e_j(y)\right| = O(\mu^{k-3}s'^3)$.
		From (\refeq{eq1}) and (\refeq{eq2}) we get
		\begin{align}\frac12\mu^{k-2}\binom{d-2}{k-2}\cdot\frac{d}{d-1}s^2  \le \binom{d}{k}\mu^k - S_k+ O(\mu^{k-3}s^3). \label{eq3}	\end{align}	
The hypothesis implies
$\binom{d}{k}\mu^k -
 S_k<\binom{d}{k}\mu^k -t^{(k)}(n,d)+t.$
Combine this with \eqref{eq4} and \eqref{eq3}, 
 $s
 =O_{k,d}\left(1+\sqrt{\frac{t}{n^{k-2}}}\right)$.
If $d\mid n$ then $s
 =O_{k,d}\left(\sqrt{\frac{t}{n^{k-2}}}\right)$.
 Choose any balanced integer vector $(m_1,\ldots,m_d)$ with sum $n$.
Since $|m_i-n/d|\le1$, this gives
\[
 \max_i|x_i-m_i|
 \le\max_i|x_i-n/d|+1
 =O_{k,d}\left(\sqrt{\frac{t}{n^{k-2}}}\right).
\]
Here the additive constant is absorbed using the preceding lower
bound on $t$.
	\end{proof}
	
\begin{lem}\label{lem-integer-balancing}\label{lem-count-transfer}Let  $d > k \ge 3$ be integers, and let $L\in \mathcal{F}(k,d)$. 
Let $K=K^{(k)}(V_1,\ldots,V_d)$ and
$t=t^{(k)}(n,d)-e(K)$. 
If all $|V_i|\ge n/(2d)$ and $t=o(n^{k})$, we can balance them in
$s=O_{k,d}(\sqrt{t/n^{k-2}})$ moves. For any edge set $B$, we can
choose distinct moved vertices meeting at most $O_{k,d}(|B|s/n)$
edges of $B$.
More generally, let $K,K'$ differ by $s$ vertex moves. If $B_0$ avoids moved vertices,
then
\[
\mathcal N_L(K'\cup B_0)\le(1+O_L(s/n))\mathcal N_L(K\cup B_0).
\]

\end{lem}
\begin{proof}
Write $n_i=|V_i|$ and $n_i=n/d+y_i$, so $\sum_{i\in[d]} y_i=0$.
The number of crossing edges is \[e(K)=\sum_{1\le i_1<\cdots<i_k\le d}
 n_{i_1}\cdots n_{i_k}.
\] Since $t=o(n^{k})$, by Lemma \ref{lem-st}, we have $|y_i|=O_{k,d}(\sqrt{t/n^{k-2}}).$

Order the initial sizes $|V_1|,\ldots,|V_d| $ from largest to smallest.
Move vertices only from a part of size large than $n/d$ above  to a part of size  small than $n/d$.
Then we can move at most $\sum_{i\in[d]} |y_i|\le  O_{k,d}(\sqrt{t/n^{k-2}})$ vertices balance them.

Since all $|V_i|\ge n/(2d)$,
every part always has at least $n/(2d)$ initial vertices not  moved.
On the other hand, 
since $\sum_{x\in V(K)} d_B(x)=k|B|$, there exist vertex $v$ with
$d_B(v)\le2dk|B|/n$.
Choose such a vertex at every move.
The number of edges of $B$ meeting at least one chosen vertex is
at most the sum of these degrees, which is $O_{k,d}(|B|s/n)$.
Repeated counting of an edge is not change this upper bound.

 Fix the exact non-crossing edge set $S\subseteq B_0$ of a copy of $L$ in $K'$.
Every vertex in $S$ keeps its part since $B_0$ avoids the moved vertices.
Each part has linear size and changes by at most  $s=o(n)$, use Lemma \ref{lem-Cs} for  such  $S$, $C(n_1,\dots,n_d,S;L)\le (1+O_L(s/n))C(n,S;L)$. 
 Sum  over exact
sets $S$ gives
\[
 \mathcal N_L(K'\cup B_0)
 \le(1+O_L(s/n))\mathcal N_L(K\cup B_0).
\]

\end{proof}
\begin{lem}  
\label{lem-mixed-sampling}
Let $G_0$ be an $n$-vertex $k$-graph and $B_0\subseteq E(G_0)$.
Let $\mathcal A$ be a family of missing edges with
$|\mathcal A|= a_0n^k$, where $a_0>0$ is fixed. Choose a uniform $u$-element subset
$B_1\subseteq\mathcal A$.
Then the expected numbers of copies of $L$ in $G_0\cup B_1$ using
at least two edges of $B_1$, and using both an edge of $B_0$
and an edge of $B_1$, are respectively
\[
O_{L,a_0}(u^2n^{l-2k}),\qquad
O_{L,a_0}(u|B_0|n^{l-2k}).
\]
\end{lem}
\begin{proof}
We  next count labelled copies, 
which changes the bound above by at most the fixed factor $|\operatorname{Aut}(L)|$.
There are $O_L(n^l)$ labelled potential copies which
 choose two of its edges whose images lie in $\mathcal A$.
For distinct edges $f_1,f_2\in\mathcal A$, uniform sampling  gives
\[
 \Pr(f_1,f_2\in B_1)
 =\frac{\binom{|\mathcal A|-2}{|B_1|-2}}{\binom{|\mathcal A|}{|B_1|}}
 =\frac{|B_1|(|B_1|-1)}{|\mathcal A|(|\mathcal A|-1)}=O_{a_0}(u^2/n^{2k})
\]
when $|B_1|\ge2$. The probability is zero otherwise. 
For a fixed potential copy, take a union bound over the choices of
two of its edges. If the copy uses at least two selected edges,
at least one of these events occurs. Requiring its other edges to
be present can only reduce its probability.
Summing over all potential copies gives the first bound.
It is zero when $u<2$.

For the second bound, fix $e\in B_0$.
To count labelled potential copies containing $e$, choose which edge
of $L$ maps to $e$. Map its $k$ labelled vertices in at most $k!$ ways.
There are then at most $n^{l-k}$ maps for the other vertices.
Thus there are $O_L(n^{l-k})$ potential copies containing $e$.
Each has at most $e(L)$ edges with image in $\mathcal A$.
A fixed such edge is selected with probability
$|B_1|/|\mathcal A|=O_{a_0}(u/n^k)$.
Sum over these edges, the copies, and $e\in B_0$.
This gives $O_{L,a_0}(u|B_0|n^{l-2k})$.
A copy containing several  edges of $B_0$ or edges of  $B_1$ may have been counted
more than once. This is harmless because we need an upper bound.
\end{proof}
\section{The universal equality range}\label{s3}
We prove Theorems~\ref{thm-tlnq1} and ~\ref{thm-ht} by counting
copies and comparing complete partitions. We assume that $d>k\ge3$ be integers and  $L\in \mathcal{F}(k,d)$. and
$L\in\mathcal F(k,d)$.
The following construction gives an exact result in the smaller
range relevant to Theorem~\ref{thm-ht}. It extends the fixed-vertex
construction for expanded cliques in the proof of Mubayi's
Theorem 7~\cite{Mubayi2013}.

\begin{lem}
\label{lem-anchored}\label{thm-exact-t}
Let $K=K^{(k)}(V_1,\ldots,V_d)$, fix $i$ and $y\notin V_i$.
Let $B$ consist of edges containing $y$, two vertices in $V_i$,
and at most one vertex in each other part. Then
\[
\mathcal N_L(K\cup B)=\sum_{e\in B}\mathcal N_L(K\cup\{e\}).
\]
For every sufficiently large $n$
and $1\le q\le\rho n^{k-1}$ with $0<\rho<1/(2d^{k-1})$, we have $t_L(n,q)=qc(n,L)$.
\end{lem}
\begin{proof}
Let $F$ be the original graph of $L$. Note that $K$ is $L$-free, otherwise,  any copy would give a proper $d$-coloring of $F$.
Thus every copy $L$ in $K\cup B$ uses at least one edge of $B$.

We next show that it cannot use two edges. Consider a copy of $L$ in $K\cup B$.
The parts give a $d$-coloring of the original vertices of $F$, and this coloring is not proper. Choose an edge $uv$ whose endpoints
lie in the same part. Its expansion belongs to $B$.
Since the only repeated part on that edge is $V_i$, so $u$ and $v$
map to its two vertices in $V_i$.
All other vertices on this expanded edge represent its expanded vertices.
Since $y\notin V_i$, the vertex $y$ is an expanded vertex of this
edge. It has degree one in the copy. Every edge of $B$ contains
$y$, so a second edge of $B$ in the copy would give $y$ degree at least two.
This is impossible. 

Any  copy in $K\cup B$ therefore uses exactly one edge of $B$. For each $e\in B$, the copies using that edge are exactly the copies
in $K\cup\{e\}$. These families are disjoint and include all copies.
It follows that
\[
 \mathcal N_L(K\cup B)=\sum_{e\in B}\mathcal N_L(K\cup\{e\}).
\]

Now take $K=T^{(k)}(n,d)$. Fix an ordered simple pattern that attains
$c(n,L)$, such as $(2,1,\dots,1,0\dots,0)$.
Choose $y$ in one of its  parts,  such as $V_2$.
Restrict the $B$  to this pattern. The number of edges of this pattern is 
\[
 \binom{|V_1|}{2}
 \prod_{\substack{3\le i_1\le \cdots\le i_{k-3}\le d}}
 |V_{i_1}|\cdots |V_{i_{k-1}|}
 =\left(\frac{1}{2d^{k-1}}+O_L(n^{-1})\right)n^{k-1}.
\]
Choose any $q$ edges from this family. Each has 
$c(n,L)$ copies of $L$, so the resulting graph has exactly $q c(n,L)$ copies.

By Theorem~\ref{extremal-graph}, every graph in the definition
of $t_L(n,q)$ contains $T^{(k)}(n,d)$ and has $q$ added edges.
For each added edge, the copies using it and no other added edge
number at least $c(n,L)$. 
Thus $t_L(n,q)\ge qc(n,L)$, which proves equality.
\end{proof}

\begin{proof}[Proof of Theorem~\ref{thm-tlnq1}]
By Theorem~\ref{extremal-graph}, the extremal graph is
$T=T^{(k)}(n,d)$. We know that
$t_L(n,q)\ge qc(n,L)$. 
For the upper bound, fix an ordered pattern that attains $c(n,L)$.
Let $\mathcal A$ contain all edges of this pattern.
If the pattern is $(a_1,\ldots,a_d)$, its size is exactly
$\prod_i\binom{|V_i|}{a_i}$.
The sum of $a_i$ is $k$, and every balanced part has linear size.
Thus $\prod_i\binom{|V_i|}{a_i}=\Theta_L(n^k)$.
Choose $\delta_0$ so that
$q\le\delta_0n^k$ implies $q\le|\mathcal A|$ for large $n$.

Take a uniform subset $B\subseteq \mathcal A$ of size $q$. By Lemma~\ref{lem-mixed-sampling}, at least one choice of $B$ such that the  number of copies using at
least two edges of $B$ is $O_L(q^2n^{l-2k})$.
No copy uses zero added edges, since $T$ is $L$-free. The  number of copies using exactly one edge of $B$ is $qc(n,L)$.
This choice proves the upper bound.
Since $c(n,L)=\Theta_L(n^{l-k})$, $q^2n^{l-2k}=O_L(q/n^k)qc(n,L)$. The relative error $q/n^k$ tends to zero when $q=o(n^k)$.
\end{proof}

\begin{proof}[Proof of Theorem~\ref{thm-ht}]
Let $H$ be a $k$-graph with $t^{(k)}(n,d)+q$ edges and $\N_L(H)=h_L(n,q)$. Let $q\le \delta_1n^{k-1}$, by Theorem~\ref{thm-tlnq1},  
\[
 \mathcal N_L(H)\le t_L(n,q)=O_L(qn^{l-k})=O_L(\delta n^{l-1}).
\] 
Choose $K(V_1,\cdots,V_d),B,M$ as in Remark \ref{remark-mb}, 
then all part sizes are $n/d+o(n)$.
And $|B|=O_L(q)$ by Lemma \ref{lem-b}.
Let
$
 t:=t^{(k)}(n,d)-e(K)\ge0$.
 Comparing the edge counts of $H$ and $K$, \eqref{eq-mb} gives
$ |B|-q=t+|M|.$
We show that $t+|M|$ must be zero.

Consider the $k$-graph $K\cup B$.
Every copy contains an edge $f\in M$ is also contains an edge $e\in B$, because $K$ is $L$-free.
And $e,f$ meet in at most one  vertex since any two edges in $L$ meet in at most one vertex.
If they meet in one vertex, their union fixes $2k-1$ vertices of a copy.
If they are disjoint, their union fixes $2k$ vertices.
In either case the number of copies containing the pair $\{e,f\}$ is at most
$C_1n^{l-2k+1}$ for a constant $C_1>0$. Sum over $B\times M$ to obtain
\[0\le\mathcal N_L(K\cup B)-\mathcal N_L(H)
 \le C_1|B||M|n^{l-2k+1}.\]
Indeed, a copy may contain several such pairs.

Suppose $t+|M|>0$.
Since $t\le|B|=O_L(q)$ and $q\le\delta n^{k-1}$,
Lemma~\ref{lem-integer-balancing} balances the partition in
$s=O_{L,d}(\sqrt{t/n^{k-2}})=O_L(\sqrt n)$ moves.
When $t=0$, take $s=0$ and keep the initial partition.
Choose the moved vertices by the degree rule in that lemma. Denote the balanced crossing graph $K'$.
Suppose that  $r$ is the number of bad edges meeting these vertices, then
 $r\le O_L(\frac{|B|s}{n})
\le O_L(|B|t/n^{k-1})
\le O_L(\delta t)$.
Take $\delta$ smaller so that  $r\le t+|M|$.
Delete exactly $t+|M|$ bad edges of $B$, including all such $r$ affected edges.
This is possible because $t+|M|=|B|-q\le|B|$.
Let $B_0$ be the remained bad edges.
Then $|B_0|=q$, and every vertex of every  edge of $B_0$ keeps its part.
In particular, the graph $H':=K'\cup B_0$ has exactly $t^{(k)}(n,d)+q$ edges.
For each deleted bad edge $e$, 
Lemma~\ref{lem-cs} gives at least $\Theta_L(n^{l-k})$ copies
in $U$ using $e$ as their only bad edge.
Thus there exist constant $C_0>0$ such that
\[
 \mathcal N_L(K\cup B_0)
 \le\mathcal N_L(K\cup B)-C_0(t+|M|)n^{l-k}.
\]

Count transfer under the $s$ moves using Lemma \ref{lem-count-transfer}, this  gives  exist a constant $C_2>0$, 
\begin{align*}
 \mathcal N_L(H')-\mathcal N_L(H)
 \le C_1(|B||M|n^{l-2k+1})-C_0(t+|M|)n^{l-k}+
   C_2(|B|s n^{l-k-1}).
\end{align*}
Here the first error comes from  missing edges. The second error 
comes from  deleted edges, and
the third comes from moving vertices after the deletion.
We bound them separately:
\[
 C_1|B||M|n^{l-2k+1}
 =\frac{|B|}{n^{k-1}}|M|n^{l-k}
 \le C_1\delta |M|n^{l-k}.
\]
Using $s= O_{k,d}(\sqrt{t/n^{k-2}})$,
\[
 C_2|B|s n^{l-k-1}
 \le O_L(C_2\frac{|B|}{n^{k-1}}t n^{l-k})
 \le O_L(C_2\delta t n^{l-k}).
\]
Choose $\delta$ so that their combined coefficient is less than $C_0/2$.
Then $\mathcal N_L(H')<\mathcal N_L(H)=h_L(n,q)$ since $t+|M|>0$.
This contradicts minimality of $h_L(n,q)$.

Therefore $t+|M|=0$, and so $t=0$ and $M=\varnothing$.
That means $K$ is balanced and  $H$ contains $T^{(k)}(n,d)$.
Thus $h_L(n,q)=\N_L(H)\ge t_L(n,q)$.
The definition always gives $h_L(n,q)\le t_L(n,q)$.
Therefore $h_L(n,q)=t_L(n,q)$.
Finally take $\delta<1/(2d^{k-1})$.
Lemma~\ref{lem-anchored} gives $h_L(n,q)=t_L(n,q)=qc(n,L)$.
This also proves $q_L(n)=\Omega_L(n^{k-1})$.
\end{proof}

\section{Thresholds of order $n^{k-1}$}\label{s4}
We give a sufficient condition for $q_L(n)=\Theta_L(n^{k-1})$.
The construction moves one vertex between balanced parts and
adds edges through a fixed vertex.
We then apply this result to expansions of $F_{s,t,d}$.
 The next Theorem  gives the Theorem \ref{thm-anchor-upper}.
\begin{thm}Let  $d > k \ge 3$ be integers, and let $L\in \mathcal{F}(k,d)$. 
Suppose
$ |\zeta_L|>\frac{4d(d-k+1)}{k-2}\alpha_L$.
Let $\theta_1=2\alpha_L\binom{d-2}{k-2}/(d^{k-2}|\zeta_L)|$ and $\theta_2=\binom{d-2}{k-3}/(2d^{k-1})$.
For all sufficiently large $n$, if $q=\theta n^{k-1}$ with $\theta_1 <\theta<\theta_2$, then
\[
h_L(n,q)<q c(n,L)\le t_L(n,q).
\]
In particular, $q_L(n)=\Theta_L(n^{k-1})$.
\end{thm}
\begin{proof}
Since
 $\binom{d-2}{k-2}/\binom{d-2}{k-3}
 =\frac{d-k+1}{k-2}$, $\theta_1\le \theta_2$.
The range of $q$ is nonempty. Suppose that $U_1,\ldots,U_d$ be the  balanced parts of $T^{(k)}(n,d)$.
Choose an ordered pattern $\boldsymbol a$  that attains $c(n,L)$, and  let $i$ be its
doubled part. 
Suppose we interchange parts $i$ and $j$ in the pattern.
By Lemma~\ref{lem-cs}, the change in the count is
\[
 \zeta_L(n_j-n_i)n^{l-k-1}+O_L(n^{l-k-2}).
\]
If $\zeta_L<0$, minimality of $c(n,L)$ forces $n_i$ to be a largest part.
If $\zeta_L>0$, it forces $n_i$ to be a smallest part.
For example, if $\zeta_L<0$ and some $n_j>n_i$, then balanced sizes
give $n_j-n_i=1$. 

For $\zeta_L<0$, move one vertex from another part $j$ into $i$.
For $\zeta_L>0$, move one vertex from $i$ into another part $j$.
In either case, by Lemma \ref{lem-st}, the loss of crossing edges is
\[\tau\le 2\binom{d-2}{k-2}({n\over d})^{k-2}+O_{k,d}(n^{k-3}).
\]
The loss is positive even if the initial two sizes are equal.
This proves both bounds on $\tau$.
Let  $K'$ be its new
complete $d$-partite $k$-graph. Thus $e(K')=t^{(k)}(n,d)-\tau$.

Fix $y$ outside part $i$. Take edges containing $y$, two vertices
of part $i$, and $k-3$ vertices in distinct other parts.
The number of these edges in $K'$ is 
\[{|U_i|-1\choose 2}
\sum_{\substack{S\subseteq[d]\setminus\{i,j\}\\|S|=k-3}}
\prod_{h\in S}|U_h|=
 \left(1+O_L(n^{-1})\right)\binom{d-2}{k-3}({n\over d})^{k-1}/2=\left(1+O_L(n^{-1})\right)\theta_2n^{k-1}.
\]
It is larger
than $q+\tau$ for large $n$ since $\tau=O_L(n^{k-2})$ and $q=
\theta n^{k-1}$ with $\theta<\theta_2$. 
Choose $q+\tau$ such edges and call the set $B$.

Let $H=K'\cup B$, then  $e(H)=t^{(k)}(n,d)+q$. By Lemma~\ref{lem-anchored}, every copy in $H$ uses one edge of $B$.
 Using Lemma \ref{lem-cs} for each $e\in B$, we have \begin{align}\label{eq-anchor-saving}\N_L(H)=\sum_{e\in B}\mathcal N_L(K'\cup\{e\})=|B|(c(n,L)-|\zeta_L|n^{l-k-1}+O_L(n^{l-k-2})).\end{align}
Since 
$c(n,L)=\alpha_Ln^{l-k}+O_L(n^{l-k-1})$ and $q=
\theta n^{k-1}$ with $\theta>\theta_1$, we get
\begin{align*}
 \mathcal N_L(H)-qc(n,L)
 &=-|\zeta_L|q n^{l-k-1}+\tau\alpha_L n^{l-k}
   +O_L(\tau n^{l-k-1}+qn^{l-k-2})\\
 &\le\bigl(-|\zeta_L|\theta+2\binom{d-2}{k-2}d^{-(k-2)}\alpha_L\bigr)n^{l-2}
   +O_L(n^{l-3})<0.
\end{align*}
On the other hand, Theorem~\ref{thm-tlnq1} gives
$t_L(n,q)\ge qc(n,L)$. Hence $q_L(n)<q=O_L(n^{k-1})$.
Theorem~\ref{thm-ht} gives the  lower bound $q_L(n)\ge O_L(n^{k-1})$.
\end{proof}


\begin{proof}[Proof of Corollary \ref{cor-multipartite-threshold}] Let $uv$ be the added  edge  in the part of size $s$ of $F_{s,t,d}$. 
The endpoints $u,v$ and one vertex from each other part form a $K_{d+1}$, so $\chi(F_{s,t,d})\ge d+1$. Also, $d+1$ colors suffice: use two colors on the first part with different colors on $u,v$, use one new color on each other part.

Deleting $uv$ leaves a complete $d$-partite graph. Now delete a different edge. We can still choose $u,v$ and one vertex from each other part so that the deleted edge is avoided. The chosen vertices still span $K_{d+1}$. Hence $uv$ is the only critical edge.

Denote $l:=|V(L)|$. The expansion has $s+(d-1)t$ original vertices and $(k-2)e(F_{s,t,d})$ new vertices.  Since $e(F_{s,t,d})-1=s(d-1)t+\binom{d-1}{2}t^2$,  \[
 l-k=s+(d-1)t-2
 +(k-2)\left(s(d-1)t+\binom{d-1}{2}t^2\right).
\] Every proper $d$-coloring of $F_{s,t,d}-uv$ uses its original parts as color classes. To see this, there are $d$ nonempty parts and only $d$ colors. Consider the leading polynomial $C_1$. The fixed pair in part $1$ receives $u,v$, and there are $s-2$ free original vertices in that part. If an original edge has an endpoint of color $1$, its expansion vertex cannot lie in part $1$. An edge avoiding color $1$ has $d-2$ allowed colors for its expansion vertex, one of which is $1$. There are $\binom{d-1}{2}t^2$ such edges, and each contributes an average $(k-2)/(d-2)$ new vertices to the class. Thus the  mean number of free vertices assigned to part $1$ is $$s-2+\frac{{k-2}}{d-2} \binom{d-1}{2}t^2=s-2+\frac{(k-2)(d-1)}{2}t^2.$$
The mean is taken over the labelled color assignments in the
counting proof. Each assignment has the same weight at the equal
vector, so  differentiating a monomial at $(1/d,\ldots,1/d)$ multiplies its value by $d$ times its exponent.
Therefore
\[
 \partial_1C(1/d,\ldots,1/d)
 =d\alpha_L\left(s-2+\frac{(k-2)(d-1)}{2}t^2\right).
\]

Fixing the added $k$-edge, 
there is $l-k$ free vertices. 
The other $d-1$ variables have equal mean exponents by symmetry.
Therefore
\[
\partial_2C(1/d,\ldots,1/d)
=d\alpha_L\frac{l-k-(s-2+\frac{(k-2)(d-1)}{2}t^2)}{d-1}.
\]
 This gives 
$\zeta_L=d\alpha_L\bigl((s-t)(1-(k-2)t)-2\bigr)$.
Apply Theorem~\ref{thm-anchor-upper}
, this proves \eqref{eq-multipartite-zeta}.
\end{proof}

\section{A sufficient condition for a larger equality range}\label{s5}
In this section, 
we prove Theorem~\ref{thm-stationary-link}.
We first define the link densities and count copies through a
fixed vertex, and 
 then use that condition to control the bad edges at each vertex.
The assumption $\zeta_L=0$ will be used when we balance the parts.

The \emph{link} of a vertex $x$ in a $k$-graph $H$ is
\[
 \operatorname{lk}_H(x)
 =\left\{S\in\binom{V(H)\setminus\{x\}}{k-1}:
                  S\cup\{x\}\in E(H)\right\}.
\]
Fix a partition $V_1,\ldots,V_d$ of the vertices other than $x$.
The pattern of a $(k-1)$-set $S$ is the vector
 $\rho(S)=(|S\cap V_1|,\ldots,|S\cap V_d|)$. 
The set of possible patterns is
\[
 \mathcal P_{k-1,d}
 =\left\{(a_1,\ldots,a_d)\in\mathbb Z_{\ge0}^d:
                         \sum_{i=1}^d a_i=k-1\right\}.
\]
Write $n_i=|V_i|$ and $n=\sum_{i\in[d]}n_i$.
There are
$\prod_{i\in[d]}\binom{n_i}{a_i}$ sets of pattern $\boldsymbol a$.
Let $w_{\boldsymbol a}(x)$ denote the \emph{density} of $(k-1)$-sets
of pattern $\boldsymbol a$ in $\operatorname{lk}_H(x)$, thus 
\[
 w_{\boldsymbol a}(x
 )
 =\frac{|\{S\in\operatorname{lk}_H(x):\rho(S)=\boldsymbol a\}|}
             {\prod_i\binom{n_i}{a_i}}.
\]
All parts considered below have linear size, so the denominator
is positive for sufficiently large $n$.
Let vector
$\boldsymbol w(x)=(w_{\boldsymbol a}(x))_{\boldsymbol a\in\mathcal P_{k-1,d}}$
 denote the density of link
edges of each pattern.
For any vector
$\boldsymbol w\in[0,1]^{\mathcal P_{k-1,d}}$, define
\begin{align}\label{eta}
\ell(\boldsymbol w)=d^{-(k-1)}
 \sum_{\boldsymbol a\in\mathcal P_{k-1,d}}
 \frac{w_{\boldsymbol a}}{\prod_i a_i!},\qquad
\eta(\boldsymbol w)=\ell(\boldsymbol w
 )-\binom{d-1}{k-1}d^{-(k-1)}.
\end{align}
We next define the polynomial that counts copies through $x$.
Fix $v\in V(L)$.
Consider all maps
$\varphi:V(L)\setminus\{v\}\to[d]$
such that every edge of $L$ avoiding $v$ receives $k$ distinct colors.
For $e\ni v$, let
$\rho(\varphi(e\setminus\{v\}))$ denote the number of vertices
of each color in $e\setminus\{v\}$.
Define the \emph{rooted counting polynomial} by
\begin{align}\label{eq-link-polynomial}
R_L(\boldsymbol w)=
\frac{1}{|\operatorname{Aut}(L)|d^{l-1}}
\sum_{v,\varphi}\prod_{\substack{e\in E(L)\\v\in e}}
w_{\rho(\varphi(e\setminus\{v\}))}.
\end{align}
Here the sum is over all such pairs $(v,\varphi)$.

After adding one vertex with link densities $\boldsymbol w$
to a  balanced $d$-partite $k$-graph,
$R_L(\boldsymbol w)$ is the leading coefficient of its copy
count as   following.

\begin{lem}\label{lem-link-count}
Let  $d > k \ge 3$ be integers, and let $L\in \mathcal{F}(k,d)$ with $l$ vertices. Add a vertex $x$ to $T^{(k)}(n,d)$ such that 
$lk_H(x)$ in the resulting graph $H$  has density $w_{\boldsymbol a}$
for  each pattern $\boldsymbol a\in \mathcal P_{k-1,d}$.
Then 
\[
e(H)=t^{(k)}(n,d)+\ell(\boldsymbol w(x))n^{k-1}+O_{k,d}(n^{k-2}),
\qquad
\mathcal N_L(H)=R_L(\boldsymbol w(x))n^{l-1}+O_L(n^{l-2}).
\]
\end{lem}
\begin{proof}
Let $n_i$ be the  part sizes of $T^{(k)}(n,d)$.
For every  pattern $\boldsymbol a=(a_1\dots,a_d)\in \mathcal P_{k-1,d}$,
there are
\[
 w_{\boldsymbol a}(x)\cdot\prod_{i\in[d]}\binom{n_i}{a_i}
 =w_{\boldsymbol a}
 (x)\frac{(n/d)^{k-1}}{\prod_i a_i!}+O_{k,d}(n^{k-2})
\]
edges of pattern $\boldsymbol a$ in $H$.
There are only finitely many patterns, depending on $k,d$.
Sum over all patterns of $\mathcal P_{k-1,d}$.
This proves the edge formula.

Since  $T^{(k)}(n,d)$ is $L$-free, so every copy uses an added edge that contains $x$.
Fix a vertex $v\in V(L)$ that maps to $x$.
Fix also a map $\varphi$ as in \eqref{eq-link-polynomial}.
For an edge $e$ that contains $v$ in $L$, suppose that $\rho(\varphi(e\setminus\{v\}))=\boldsymbol b=\{b_1,\dots,b_d\}$.
There are
$w_{\boldsymbol b}(x)\cdot \prod_{i\in[d]}\binom{n_i}{b_i}$
edges   has this pattern in $lk_H(x)$. For every such edge, there are
$\prod_{i\in[d]} b_i!$ maps of the remain labelled vertices of $e\setminus\{v\}$
that respect $\varphi$. Thus the number of labelled choices for this group is
\[
w_{\boldsymbol b}(x)\cdot\prod_{i\in[d]}\binom{n_i}{b_i}\cdot\prod_{i\in[d]} b_i!=w_{\boldsymbol b}(x)({n\over d})^{k-1}+O_L(n^{k-2}).
\]
Any two edges of $L$ intersect  at most one vertex, so the sets $e\setminus\{v\}$ for $e\ni v$ are pairwise disjoint. 
We may first choose their images independently. Their union contains exactly $d_L(v)(k-1)$ labelled vertices.  Each of the remaining $l-1-d_L(v)(k-1)$
labels has $n/d+O(1)$ choices in the  parts specified by $\varphi$. Before excluding repeated images, the number of choices is
\[
 (n/d)^{l-1}
       \prod_{e\ni v}w_{\rho(\varphi(e\setminus\{v\}))}
       +O_L(n^{l-2}).
\]

We now exclude maps that identifies two vertices.
Such a map has at most $l-2$
free images, so
there are at most $n^{l-2}$ choices.
There are at most $\binom l2$ such pairs.
Thus all such choices number $O_L(n^{l-2})$.
Sum over $v$ and $\varphi$, then divide by
$|\operatorname{Aut}(L)|$. This gives the copy formula.
\end{proof}


For $i\in[d]$,  define
\[
 p_i(\boldsymbol w)
 =\sum_{\substack{\boldsymbol a\in\mathcal P_{k-1,d}\\a_i>0}}
       \frac{w_{\boldsymbol a}}{\prod_j a_j!}.
\]
Then $(1+O(n^{-1}))p_i(\boldsymbol w)(n/d)^{k-1}$ is  the number of link sets meeting $V_i$.
Each such set is counted once, even if it has several vertices in $V_i$.
Since every ${\prod_j a_j!}$ is nonnegative,
\[
 p_i(\boldsymbol w)=0
 \quad\Longleftrightarrow\quad
 w_{\boldsymbol a}=0
 \text{ for every $\boldsymbol a$ with $a_i>0$}.
\]


\begin{proof}[Proof of Theorem~\ref{thm-stationary-link}]
Assume that $\zeta_L=0$ and that \eqref{eq-link-rigidity} holds.
For the estimates below, first allow
 $1\le q\le n^{k-1/2}.$
Let $H$ be a $n$-vertex $k$-graphs with $t^{(k)}(n,d)+q$ edges and $\N_L(H)=h_L(n,q)$.
We will choose a fixed constant $\delta_L>0$ at the end.
All constants may depend on $L,k,d$.

Recall that  $c(n,L)=\alpha_Ln^{l-k}+O_L(n^{l-k-1})$,
and $q/n^k\le n^{-1/2}$. Theorem~\ref{thm-tlnq1} gives
\begin{equation}\label{eq-s5-initial-upper}
 \mathcal N_L(H)
 \le t_L(n,q)
 \le qc(n,L)+O_L(q^2n^{l-2k})
 =O_L(qn^{l-k}).
\end{equation}
Choose the partition $V_1,\ldots,V_d$ and maximizes $\Phi=\sum_{j=1}^k jh_j$ from Remark~\ref{remark-mb}. Let 
\[
 K=K^{(k)}(V_1,\ldots,V_d),\qquad
 B=E(H)\setminus E(K),\qquad
 M=E(K)\setminus E(H), \quad t=t^{(k)}(n,d)-e(K).
\]
Thus $B$ is the set of bad edges and $M$ is the set of
missing crossing edges.
And the Remark gives $n_i:=|V_i|=n/d+o(n)$, $|B|=o(n^k)$, $|M|=o(n^k)$, and
\begin{equation}\label{eq-s5-edge-identity}
 e(K)-|M|+|B|=t^{(k)}(n,d)+q,
 \qquad
 |B|-q=t+|M|.
\end{equation} 
We next prove that both $t$ and $|M|$ are zero, and thus $H=T^{(k)}(n,d)\cup B$.

\begin{clm}\label{clm-s5-local-bound}
For $x\in V_i$, let  $Y_i(x)=\{e\in E(H):x\in e,\ |e\cap V_i|\ge2\}$. Then  $Y_*:=\max_{i\in[d],x\in V_i}|Y_i|=o(n^{k-1})$.
\end{clm}

Fix $x\in V_i$, and let $\boldsymbol w(x)$ be its link vector. The $Y_i(x)$ need not contain every bad edge through $x$. By the same  count in Lemma~\ref{lem-link-count}, we have
\[
 |Y_j(x)|
 \le d^{-(k-1)}p_j(\boldsymbol w(x))n^{k-1}
  +o(n^{k-1}) \text{ for each } j\in[d].
\]
For each $j\in[d]$, count the sets in
$\operatorname{lk}_H(x)$ that meet $V_j\setminus\{x\}$.
Thus maximality of $\Phi$ gives
\[
 |\{S\in\operatorname{lk}_H(x):S\cap V_i\ne\varnothing\}|
 \le
 |\{S\in\operatorname{lk}_H(x):S\cap V_j\ne\varnothing\}|
 \text{ for each } j\in[d].
\]
Otherwise, moving the vertex $x$ from part $i$ to part $j$ will increase 
$\Phi$: If the edge avoids $x$, then removing $x$ from the parts
does not change $\Phi$. If the edge contains $x$. Every set $S$ with $S\cap V_i\ne\varnothing$ and $S\cap V_j=\varnothing$  increases its
contribution to $\Phi$ by exactly one.
Every set $S$ with $S\cap V_i\ne\varnothing$ and $S\cap V_j\ne\varnothing$  does not change $\Phi$.
Therefore, by the same  count in Lemma~\ref{lem-link-count}, 
\begin{equation}\label{eq-s5-minimal-part}
 p_i(\boldsymbol w(x))\le p_j(\boldsymbol w(x))+o(1)
 \text{ for each } j\in[d].
\end{equation}

Let $\mathcal N_L(x;H)$ be the number of copies containing $x$.
Lemma~\ref{lem-link-count} counts
 $R_L(\boldsymbol w(x))n^{l-1}+o(n^{l-1})$
potential copies through $x$ in $H\cup M$.
For a fixed edge $f\in M$, there are  at most $n^{l-k-1}$   copies counted above contins $f$ and $x$. 
Indeed, if  $x\in f$ then every such copy contains a edge of  $B$ since $K$ is $L$-free, and if $x\notin f$ then every such copy fixes the $k + 1$ vertices
$(x\cup V (f ))$.
There are then at most $n^{l-k-1}$ choices for the other vertices.
Since $|M|n^{l-k-1}=o(n^{l-1})$, summing over $M$ gives
\begin{equation}\label{eq-s5-root-lower}
 \mathcal N_L(x;H)
 \ge R_L(\boldsymbol w(x))n^{l-1}
             -o(n^{l-1}).
\end{equation}

Now replace  all bad edges of $H$   containing $x$ by  the missing edges 
containing $x$. Only the degree of $x$ has changed, so
\[
 r:=e(H)-e(J)=d_B(x)-d_M(x)
   \le\eta(\boldsymbol w(x))n^{k-1}+o(n^{k-1}),
 \qquad |r|=O_L(n^{k-1}).
\]
Call the new graph $J$.
Copies avoiding $x$ are unchanged.
A copy through $x$ in $J$ must contain a bad edge away from $x$,
because $K$ is $L$-free.
Fixing that bad edge and $x$ gives
\begin{align}\label{J1}
 \mathcal N_L(J)
 \le \mathcal N_L(H)-\mathcal N_L(x;H)
                         +O_L(|B|n^{l-k-1}).
\end{align}

If $r<0$, delete $-r$ edges away from $x$. This deletion is possible and creates no new copies since
there are $\Theta_{k,d}(n^k)$ such edges in $J$.
Suppose that $r\ge 0$.  For a simple  pattern that one part is used twice
and $k-2$ other parts are used once,
There are $\Theta_{k,d}(n^k)$ edges of this pattern avoiding $x$.
And only $o(n^k)$ of them are present in $J$, because all present
edges of this pattern belong to $B$.
Therefore the set of available edges  has size
$\Theta_{k,d}(n^k)$.
Choose a uniform subset $B_1$ with $|B_1|=r$ of it and add these edges to $J$. In either case the result graph $J\cup B_1$ has exactly $e(H)$ edges.

By Theorem~\ref{thm-c} and
Lemma~\ref{lem-cs},
the copies of $L$  use just one added  edge  has count
$(\alpha_L+o(1))n^{l-k}$.
Thus these copies number at most
$r(\alpha_L+o(1))n^{l-k}$.
Every other copy in $J\cup B_1$ containing either two added edges, or a added edge and an old bad edge away from $x$.
Lemma~\ref{lem-mixed-sampling} bounds their expected number by
$ O_L\bigl(r^2n^{l-2k}+r|B|n^{l-2k}\bigr)=o(n^{l-1})$.
Indeed,
$ r^2n^{l-2k}=O_L(n^{l-2})$ and
 $r|B|n^{l-2k}
 =O_L\left(\frac{|B|}{n^k}n^{l-1}\right)=o(n^{l-1})$.
Therefore \begin{align}\label{J2}
\N_L(J\cup B_1)-r(\alpha_L+o(1))n^{l-k}-o(n^{l-1})\le \N_L(J).
\end{align}

Since $\N_L(H)\le \N_L(J\cup B_1)$, combine \eqref{J1} with \eqref{J2}, we have
\[\mathcal N_L(x;H)
 \le \alpha_L\max\{r,0\}n^{l-k}+o(n^{l-1})\le \alpha_L\max\{\eta(\boldsymbol w(x)),0\}n^{l-1}+o(n^{l-1}).\]
Combine this with \eqref{eq-s5-root-lower}, we obtain
\begin{equation}\label{eq-s5-link-approx}
 R_L(\boldsymbol w(x))
 \le \alpha_L\max\{\eta(\boldsymbol w(x)),0\}+o(1).
\end{equation}

Suppose otherwise.
Since $n\to\boldsymbol \infty$, there are $\varepsilon>0$ and a sequence of choices
$n,q,H,x\in V_i$ for which
$|Y_i(x)|\ge\varepsilon n^{k-1}$.
Choose a subsequence on which $i$ is fixed, and choose
a further subsequence on which
$\boldsymbol w(x)\to\boldsymbol w^*$.
Continuity and \eqref{eq-s5-link-approx} give
 $R_L(\boldsymbol w^*)
 \le\alpha_L\max\{\eta(\boldsymbol w^*),0\}$.
And \eqref{eq-s5-minimal-part} gives
 $p_i(\boldsymbol w^*)=\min_{j\in[d]}p_j(\boldsymbol w^*)\ge d^{k-1}\varepsilon>0,$
which contradicts \eqref{eq-link-rigidity}. This completes the proof of Claim.

We first compare $\N_L(H)$ and $\N_L(K\cup B)$. Fix $f\in M$.  Since $U$ is $L$-free, a potential copy has missing edge $f$, then it also has a bad edge  $e
\in B$ with both original endpoints in the same part. 
If $e$ and $f$ intersect, their common
vertex must be one of the two original endpoints of $e$. The other vertices of $e$ have degree one
in the copy.  
There are at most $kY_*$ possible such edges $e$
that meet $f$, and each intersecting pair $(e,f)$ occurs in at most
$O_L(n^{l-2k+1})$ copies.  Moreover, each disjoint pair $(e,f)$  occurs in at most $O_L(n^{l-2k})$ copies, 
and there are at most $|B|$ choices of $e$ for the fix $f$.
Summing the intersecting and disjoint separately gives 
\begin{equation}\label{eq-general-rigid-fill}
 0\le\mathcal N_L(K\cup B)-\mathcal N_L(H)
 \le C_L\bigl(|M|Y_*n^{l-2k+1}
                       +|M||B|n^{l-2k}\bigr)
 \le o(|M|n^{l-k}).
\end{equation}
When $M=\varnothing$, both sides of the upper bound are zero. 
Since $\mathcal N_L(K\cup B)\ge |B|(1-o(1))c(n,L)\ge O_L(|B|n^{l-k})$, combine \eqref{eq-s5-initial-upper} with \eqref{eq-general-rigid-fill} give  $|B|=O_L(\N_L(H)n^{k-l})=O_L(q)$. 
It follows that $|M|\le |B|=O_L(q)$ and $t\le |B|=O_L(q)$.

Let $m_1,\ldots,m_d$ be the ordered balanced
part sizes in Lemma~\ref{lem-integer-balancing}, and let
$A=\max_{i\in[d]}|n_i-m_i|$.
By Lemma \ref{lem-st}, $A^2=O_{k,d}(t/n^{k-2})$.
For any edge $e\in B$, consider  copies with exactly one bad edge $e$ in $K\cup B$.
Since $\zeta_L=0$, Theorem \ref{thm-c} gives
\begin{equation}\label{eq-general-rigid-transfer}
 |C(n_1,\dots,n_d,e;L)- c(n,L)|\le 
 O_L(A^2)n^{l-k-2}\le O_L(t)n^{l-2k}.
\end{equation}
Let $\N_L^2(K\cup B)$ be the number of copies in $K\cup B$ with at least
two bad edges. Thus
\[
 \mathcal N_L(K\cup B)
 =\sum_{e\in B}C(n_1,\dots,n_d,e;L)+\N_L^2(K\cup B).
\]
 Combine this with \eqref{eq-s5-initial-upper},   \eqref{eq-general-rigid-fill} and \eqref{eq-general-rigid-transfer}, since $\N_L(K\cup B)\ge \N_L(H)$, we obtain
\[
 \begin{split}
  \N_L^2(K\cup B)\le{}&-(|B|-q)c(n,L)+O_L(|B|t) n^{l-2k}+
 O_L(q^2n^{l-2k})+o(|M|n^{l-k}).
 \end{split}
\]
Since $c(n,L)=\Theta(n^{f-k})$, $|B|-q=t+|M|$, $|M|\le O_L(q)$, $t\le O_L(q)$ and $q\le n^{k-1/2}$,   
we obtain $ \N_L^2(K\cup B)\le O_L(q^2n^{l-2k})$.

Suppose for a contradiction, that $t+|M|>0$.
Use the balancing moves in Lemma~\ref{lem-integer-balancing}. The number of moves  $s$  satisfies
$s=O_{d}(A)= O_{k,d}(t/n^{k-2})$.
At most $r=O_L(|B|s/n)\le O_L(qt/n^{k-1})$  bad edges meeting at least one moved vertex.
Delete
 $b=\max\{r,t+|M|\}$
edges of $B$, including every edge meeting a moved vertex.
This is possible because $r\le |B|$ and
$t+|M|=|B|-q\le |B|$.
Let $B_0$ be the remaining edge subset of $B$, and let $K'$ be the balanced complete $d$-partite $k$-graph.
Choose a pattern attaining the exact minimum $c(n,L)$ in $K'$, its total number of
edges is $(1+o(1))\prod_{i\in[d]}\binom{n/d}{a_i}=\Theta_{k,d}(n^k)>q\ge |B_0|$.
If $r>t+|M|$, then $b=r$ and  choose a  subset $B_1$ with $|B_1|=r-(t+|M|)$ of this available family and set $H':= K'\cup B_0\cup B_1$.
Therefore $e(H')=t^{(k)}(n,d)+q$.

We next compare $\N_L(K\cup B)$ and $\N_L(H')$.
For each deleted edge, all copies having it as their only
bad edge disappear. These families are disjoint.
By \eqref{eq-general-rigid-transfer}, the decrease is at least $b(c(n,L)-O_L(t) n^{l-2k})$.
Next replace $K\cup B_0$ by $K'\cup B_0$.
For copies with exactly one bad edge of $B_0$, the change is at most
$ O_L(|B_0|A^2) n^{l-k-2}\le O_L(qt) n^{l-2k}$
by \eqref{eq-general-rigid-transfer}.
We estimate copies with at least two  bad edges of $B_0$.
Before balancing, their count is at most $\N_L^2(K\cup B)$.
As in
Lemma~\ref{lem-count-transfer}, their increase is at most
$O_L(\frac{s}{n}\N_L^2(K\cup B) \le O_L(s q^2)n^{l-2k-1}$.
Finally add $B_1$ with $
|B_1|=\max\{r-(t+|M|),0\}$.
Each  edge of $B_1$  has exactly $c(n,L)$ copies  using it as
the only bad edge in $K'\cup B_1$.
Thus these copies contribute $|B_1|c(n,L)$.
Every other  copy uses either two edges of $B_1$,
or one edge of $B_1$ and an edge of $B_0$.
Lemma~\ref{lem-mixed-sampling} bounds their expected number by
$O_L\bigl(|B_1||B_0|+|B_1|^2\bigr)n^{l-2k}\le O_L\bigl(rq+r^2\bigr)n^{l-2k}
 \le O_L(q^2t/n^{k-1})n^{l-2k}$. The combined estimate is \begin{align*}
 \N_L(K\cup B)-\mathbb E(\mathcal N_L(H'))
 \ge& b(c(n,L)-O_L(t) n^{l-2k})-O_L(qt) n^{l-2k}-O_L(s q^2)n^{l-2k-1}\notag \\&-|B_1|c(n,L)-O_L(q^2t/n^{k-1})n^{l-2k}\notag\\
 =(t+|M|)c(n,L)&-O_L((b+q)t) n^{l-2k}-O_L(s q^2)n^{l-2k-1}- O_L(q^2t)n^{l-3k+1}.\end{align*}

Therefore     \eqref{eq-general-rigid-fill} gives
\[\begin{split}
  \mathbb E(\mathcal N_L(H'))-&\mathcal N_L(H)\le-(t+|M|)c(n,L)\\
 &+o(|M|)n^{l-k}+O_L\bigl((b+q)t n^{l-2k}
             +s q^2n^{l-2k-1}
             +q^2tn^{l-3k+1}\bigr).
 \end{split}
\]
Since $b\le |B|=O_L(q)$,  $t\le t+|M|$ and $s= O_L(t/n^{k-2})$, we obtain
\[
 \frac{(b+q)t n^{l-2k}}{(t+|M|)n^{l-k}}
 \le O_L(\frac{q}{n^k}),\quad  \frac{s q^2n^{l-2k-1}}{(t+|M|)n^{l-k}}
 \le O_L(\frac{q^2}{n^{2k-1}}),\quad\frac{q^2tn^{l-3k+1}}{(t+|M|)n^{l-k}}
 \le O_L(\frac{q
 ^2}{n^{2k-1}}) .
\]
Use $c(n,L)=\alpha_Ln^{l-k}+O_L(n^{l-k-1})$.
There are a sufficiently constant $\varepsilon_n$ and  a  constant $K_L>0$ such that
\begin{equation}
 \mathbb E(\mathcal N_L(H'))-\mathcal N_L(H)
 \le (t+|M|)n^{l-k}
 \left[
  -\alpha_L+\varepsilon_n
  +K_L\left(\frac{q}{n^k}+{q^2\over n^{2k-1}}\right)\notag
 \right].
\end{equation}
The constants $\varepsilon_n$ and $K_L$ are independent of $q$ and $H$.
Choose $0<\delta_L\le1$ so that
$K_L\delta_L^2\le\alpha_L/4$.
Let $1\le q\le\delta_Ln^{k-1/2}$. Since $t+|M|>0$,  for sufficiently large $n$,  
$ \mathbb E(\mathcal N_L(H'))<\mathcal N_L(H)$.
Therefore, some choice of $B_1$  gives fewer copies than $H$. This contradicts  $\N_L(H)=h_L(n,q)$.

We conclude that $t+|M|=0$, and  so $t=0$ and $M=\varnothing$.
The equality $t=0$ means that the parts are balanced.
Thus $H$ contains $T^{(k)}(n,d)$ and  $h_L(n,q)\ge t_L(n,q)$.
On the other hand, the definitions give $h_L(n,q)\le t_L(n,q)$. Therefore $h_L(n,q)=t_L(n,q)$, and
$q_L(n)=\Omega_L(n^{k-1/2})$.
\end{proof}
\section{Two applications}\label{s6}
For $F_{s,t,d}$, let $uv$ be the added edge in its part of size $s$. We prove $q_{\K_3(F_{2,3,d})}(n)=\Omega(n^{5/2})$ for
$4\le d\le25$, and  $q_{\K_3(F_{2,5,d})}(n)=\Theta(n^2)$ for every
$d\ge4$.
\subsection{Counting formulas}\label{sec:counts}
We compute the derivative and rooted counting polynomial used below.
The polynomial needed below is
\begin{equation}\label{eq:polynomial}
 P_d(r_1,\ldots,r_d)=\sum_{i=1}^d r_i\prod_{j\ne i}r_j^3.
\end{equation}
All its coefficients are nonnegative.
Hence increasing any coordinate cannot decrease its value.

\begin{prop}\label{prop:link-count}
Let $L_d=\K_3(F_{2,3,d})$. Then
\begin{equation}\label{eq:counts}
 \frac{R_{L_d}(\boldsymbol w)}{\alpha_{L_d}}
 =\frac{\sum_iw_{ii}}{2d^2}
 +\frac{P_d(W_1+w_{11},\ldots,W_d+w_{dd})}
        {d^2(d-2)^{3(d-1)}}.
\end{equation}
Consequently,
\begin{equation}\label{eq:difference}
 \frac{R_{L_d}(\boldsymbol w)}{\alpha_{L_d}}-\eta(\boldsymbol w)
 =
 \frac{
 P_d(W_1+w_{11},\ldots,W_d+w_{dd})
 -(d-2)^{3(d-1)}\left(\sum_{i<j}w_{ij}-\binom{d-1}{2}\right)}
 {d^2(d-2)^{3(d-1)}}.
\end{equation}
\end{prop}

\begin{proof}
Write $F=F_{2,3,d}$ and $l=|V(L_d)|$.
We count labelled maps first. Every unlabelled copy corresponds
to exactly $|\Aut(L_d)|$ such maps.

Fix an added triple with two vertices in one part.
By Theorem~\ref{lem:critical}, its repeated pair must receive $u,v$.
There are two orders for this pair.
The other $d-1$ original parts can be assigned to colors in
$(d-1)!$ ways. Each of the remaining $e(F)-1$ expansion vertices
has $d-2$ choices of color.
The other $l-3$ vertices are free, and every vertex has $n/d$ choices. Thus
\begin{equation}\label{eq:alpha-count}
 \alpha_{L_d}|\Aut(L_d)|
 =\frac{2(d-1)!(d-2)^{e(F)-1}}{d^{l-3}}.
\end{equation}

We now count the rooted maps in \eqref{eq-link-polynomial}.
There are two possible types of root.

First suppose that the root is an original vertex. The root must be \(u\) or \(v\). By symmetry, assume that it is \(u\).
Indeed, deleting any other original vertex leaves the clique
formed by $u,v$ and one vertex in each remaining part.
The part from which a vertex was deleted still has at least two
vertices. Thus the remaining original graph contains $K_{d+1}$.
Its expanded edges cannot all be crossing in $d$ parts.

Delete $u$. The original vertices now form a complete
$d$-partite graph with part sizes $1,3,\ldots,3$.
Suppose that  $v$ has color $i$.
There are $(d-1)!$ ways to color the other original parts.
For an edge through $u$, fix the color $j$ of its other original
endpoint. Its expansion vertex may have any color $h\in[d]$.
Summing the corresponding link factor gives
$
 \sum_{h=1}^d w_{jh}=W_j+w_{jj}.
$
The vertex $u$ has $1+3(d-1)$ incident edges in $F$, with one neighbor of color $i$ and three of every other color.
Hence the product of the link factors is
$
 (W_i+w_{ii})\prod_{j\ne i}(W_j+w_{jj})^3.
$
Each edge not incident with $u$ gives $d-2$ choices for its
expansion color. Summing over $i$ and both roots $u,v$,
the contribution multiplied by $|\Aut(L_d)|$ is
\[
 \frac{2(d-1)!(d-2)^{e(F)-1-3(d-1)}}{d^{l-1}}
 P_d(W_1+w_{11},\ldots,W_d+w_{dd}).
\]
Divide by \eqref{eq:alpha-count}.
This gives the second term in \eqref{eq:counts}.

Next suppose that the root is an expansion vertex on an edge $ab$.
All other expanded edges are crossing.
Their endpoint colors give a proper $d$-coloring of $F-ab$.
Thus $ab$ is critical, and Theorem~\ref{lem:critical} gives $ab=uv$.
The vertices $u,v$ have the same color.
Summing over that color gives  $\sum_iw_{ii}$.
The other color assignments give
$(d-1)!(d-2)^{e(F)-1}$.
After division by \eqref{eq:alpha-count}, this contribution is
$\sum_iw_{ii}/(2d^2)$.

These two types include every possible root.
Adding their contributions proves \eqref{eq:counts}.
Subtracting \eqref{eta} proves
\eqref{eq:difference}.
\end{proof}

We prove a concavity statement.
It will let us replace arbitrary edge weights by weights in
$\{0,1\}$, with at most one exception.

\begin{lem}\label{lem:concavity}
The function $\log P_d(r_1,\ldots,r_d)$ is concave on the set where all $r_i>0$.
\end{lem}

\begin{proof}
For positive $r_i$,
\[
 P_d(r_1,\ldots,r_d)
 =\left(\prod_i r_i^3\right)\left(\sum_i r_i^{-2}\right), \quad \log P_d(r_1,\ldots,r_d)
 =3\sum_i\log r_i+\log(\sum_i r_i^{-2}).
\]
Let $a_i=r_i^{-2}/\left(\sum_j r_j^{-2}\right)$, they are positive and sum to one.
Fix any real numbers $x_1,\ldots,x_d$. Along the line with coordinates $r_i+s r_ix_i$, direct differentiation gives
\[
 \left.\frac{\partial^2}{\partial s^2}
 \log P_d(r_1+s r_1x_1,\ldots,r_d+s r_dx_d)\right|_{s=0}
 =\sum_i(-3+6a_i)x_i^2-4\left(\sum_i a_ix_i\right)^2.
\]
To obtain this formula, first differentiate
$3\sum_i\log r_i+\log(\sum_i r_i^{-2})$ along the line.
The first term contributes $-3\sum_i x_i^2$.
The second contributes $6\sum_i a_ix_i^2-4(\sum_i a_ix_i)^2$.
Since every $r_i$ is positive, every direction has the form $(r_ix_i)_i$.
It is therefore enough to show that
\begin{equation}\label{eq:quadratic-form}
 \sum_i(3-6a_i)x_i^2
 +4\left(\sum_i a_ix_i\right)^2\ge0.
\end{equation}

If every $a_i\le1/2$, then \eqref{eq:quadratic-form} holds.
Now suppose, after relabeling, that $a_1>1/2$.
For $i>1$, we have $a_i\le1-a_1<1/2$, so $3-6a_i>0$.
Apply the Cauchy--Schwarz inequality to
$a_1x_1=(\sum_i a_ix_i)-\sum_{i>1}a_ix_i$.
It gives
\[
 (a_1x_1)^2
 \le\left(1+\sum_{i>1}\frac{4a_i^2}{3-6a_i}\right)
 \left[
 \left(\sum_i a_ix_i\right)^2
 +\frac14\sum_{i>1}(3-6a_i)x_i^2
 \right].
\]
Hence the left side of \eqref{eq:quadratic-form} is at least
\[
 \left(
 3-6a_1+
 \frac{4a_1^2}{1+\sum_{i>1}4a_i^2/(3-6a_i)}
 \right)x_1^2, \ \text{where}
 \sum_{i>1}\frac{4a_i^2}{3-6a_i}
 \le\frac{4\sum_{i>1}a_i^2}{6a_1-3}
 \le\frac{4(1-a_1)^2}{6a_1-3}.
\]
The coefficient of $x_1^2$ is therefore at least
\[
 3-6a_1+\frac{4a_1^2}{1+4(1-a_1)^2/(6a_1-3)}
 =\frac{3(2a_1-1)^2}
        {1+(2a_1-1)+(2a_1-1)^2}>0.
\]
This proves \eqref{eq:quadratic-form} in both cases. Thus \(\log P_d\) has a nonpositive second derivative along every line segment in its domain. Hence \(\log P_d\) is concave.
This proves the lemma.
\end{proof}
\subsection{A family with threshold of order at least $n^{5/2}$}\label{sec:stationary}

\begin{thm}\label{thm:inequality}
Let $4\le d\le 25$ and $0\le w_{ij}=w_{ji}\le1$ for $i\ne j$.
If $\sum_{i<j}w_{ij}\ge\binom{d-1}{2}$, then
\begin{equation}\label{eq:inequality}
 P_d(W_1,\ldots,W_d)
 \ge(d-2)^{3(d-1)}
       \left(\sum_{i<j}w_{ij}-\binom{d-1}{2}\right).
\end{equation}
The inequality is strict when $\sum_{i<j}w_{ij}>\binom{d-1}{2}$.
At $\sum_{i<j}w_{ij}=\binom{d-1}{2}$, equality holds exactly
when one row is zero and all edges among the other $d-1$ vertices
have weight one.
\end{thm}

\begin{proof}
We divide the proof into three cases.

\noindent\textbf{Case 1. $\sum_{i<j}w_{ij}=\binom{d-1}{2}$.}

Note that all $W_i\ge0$. If   $W_i>0$ for every $i\in[d]$, then
	$
	P_d(W_1,\ldots,W_d)>0$ and  \eqref{eq:inequality} holds.
Assume that $W_i=0$,
	 nonnegativity gives
	\[
	w_{ij}=0\quad(j\ne i),\qquad
	\sum_{\substack{a<b\\a,b\ne i}}w_{ab}=\binom{d-1}{2}-W_i=\binom{d-1}{2}.
	\]
	There are exactly $\binom{d-1}{2}$ terms in the last sum,
	each at most one. Hence
	$
	w_{ab}=1$ for every $a<b,\ a,b\ne i.
	$
	This proves  the equality statement for \eqref{eq:inequality}.

\noindent\textbf{Case 2. $\sum_{i<j}w_{ij}=\binom{d}{2}$.}

In this case, all weights $w_{ij}=1, i<j$.  Therefore \eqref{eq:inequality} holds since
	\[
P_d(W_1,\ldots,W_d)=	P_d(d-1,\ldots,d-1)
	=d(d-1)^{3d-2}>(d-2)^{3(d-1)}(d-1).
	\]

\noindent\textbf{Case 3. $\binom{d-1}{2}<\sum_{i<j}w_{ij}<\binom d2$.}

In this case,
	\[
	W_i=\sum_{i<j}w_{ij}-\sum_{\substack{a<b\\a,b\ne i}}w_{ab}
	\ge\sum_{i<j}w_{ij}-\binom{d-1}{2}>0
	\qquad(i\in[d]).
	\]
Fix the total weight $\sum_{i<j}w_{ij}$ in this interval.
The feasible set with this total weight is compact.
Choose a minimizer $\boldsymbol w$ such that, first,  $P_d(W_1,\ldots,W_d)$ attains a minimum, and second,
	$|\{ij:0<w_{ij}<1\}|$ is minimized.
	Suppose two distinct edges $e,f$ have weights in $(0,1)$.
	Define
	\[
	w_g(s)=
	\begin{cases}
	w_e+s,&g=e,\\
	w_f-s,&g=f,\\
	w_g,&g\notin\{e,f\},
	\end{cases}
	\qquad
	\max\{-w_e,w_f-1\}\le s\le\min\{1-w_e,w_f\}.
	\]
	Every $\boldsymbol w(s)$ is feasible. Moreover,
$
	W_i(s)=W_i+s\bigl(\mathbf1_{\{i\in e\}}-\mathbf1_{\{i\in f\}}\bigr)>0.
$
	By Lemma~\ref{lem:concavity},
	$\log P_d(W_1(s),\ldots,W_d(s))$ is concave.
	Thus at least one endpoint $s_0$ of the displayed interval satisfies
	\[
	P_d(W_1(s_0),\ldots,W_d(s_0))
	\le P_d(W_1,\ldots,W_d).
	\]
	Minimality forces equality, and then at least one of
	$w_e(s_0),w_f(s_0)$ belongs to $\{0,1\}$.
	Since all other coordinates are unchanged, we have
	\[
	|\{ij:0<w_{ij}(s_0)<1\}|
	<|\{ij:0<w_{ij}<1\}|,
	\]
	contrary to the choice of $\boldsymbol w$.
	Hence the chosen minimizing vector has at most one coordinate satisfying $0<w_{ij}<1$.

	Let $G$ consist of its weight-one edges, and let $uv$ be
the possible edge with weight $t\in[0,1)$.
If there is no fractional edge, choose any missing edge as $uv$
and  take $t=0$.
	Then
	\[
	e(G)+t=\sum_{i<j}w_{ij},\qquad
	\binom{d-1}{2}\le e(G)<\binom d2.
	\]
Let $a=\deg_G(u)$, $b=\deg_G(v)$, and let
$c_1,\ldots,c_{d-2}$ be the other degrees.
	Up to a permutation, the weighted degrees are
	$
	(W_1,\ldots,W_d)=(a+t,b+t,c_1,\ldots,c_{d-2}).
$
The terms of $P_d$ indexed by $u,v$ sum to
\[
\left(\prod_{h=1}^{d-2}c_h^3\right)
\bigl((a+t)(b+t)^3+(b+t)(a+t)^3\bigr).
\]
The other terms sum to
\[
\left(\sum_{h=1}^{d-2}c_h\prod_{j\ne h}c_j^3\right)
(a+t)^3(b+t)^3.
\]
Thus the difference between the two sides of \eqref{eq:inequality} is
	\begin{align}
	A_{G,uv}(t):={}&
	\left(\prod_{h=1}^{d-2}c_h^3\right)
	(a+t)(b+t)\bigl((a+t)^2+(b+t)^2\bigr)+\left(\sum_{h=1}^{d-2}c_h
	\prod_{\substack{1\le j\le d-2\\j\ne h}}c_j^3\right)
	(a+t)^3(b+t)^3\notag\\
	&-(d-2)^{3(d-1)}
	\left(e(G)+t-\binom{d-1}{2}\right).
	\label{eq:certificate-poly}
	\end{align}
It has degree at most six.
	Write $A_{G,uv}(t)=\sum_{j=0}^6A_jt^j$.
	The binomial identity
	\begin{align*}
t^j=t^j\sum_{h=0}^{6-j}\binom{6-j}{h}t^h(1-t)^{6-j-h}=	\sum_{i=j}^6\frac{\binom ij}{\binom6j}
	\binom6i t^i(1-t)^{6-i}
	\end{align*}
	gives
	\begin{equation}\label{eq:bernstein}
	\begin{split}
	A_{G,uv}(t)=\sum_{i=0}^6B_i\binom6i t^i(1-t)^{6-i},\quad
	B_i=\sum_{j=0}^i A_j\frac{\binom ij}{\binom6j}
	\qquad(0\le i\le6).
	\end{split}
	\end{equation}
Each function $\binom6i t^i(1-t)^{6-i}$ is nonnegative on $[0,1]$.
Hence $B_i\ge0$ for all $i$ implies $A_{G,uv}(t)\ge0$ there.
Moreover, $B_6>0$ implies strict positivity for every $t>0$.
At $t=0$, the value is $B_0$.

The program in Appendix enumerates every degree
signature arising from a graph $G$ and a distinguished missing edge.
It computes $60B_i$ in \eqref{eq:bernstein} by exact integer arithmetic. For each resulting polynomial it verifies
\[
B_i\ge0\quad(0\le i\le6),\qquad B_6>0,
\qquad
B_0>0\quad\text{if }e(G)>\binom{d-1}{2}.
\]
For $0<t\le1$, \eqref{eq:bernstein} yields
	$
	A_{G,uv}(t)\ge B_6t^6>0.
	$
	If $t=0$, then $e(G)=\sum_{i<j}w_{ij}>\binom{d-1}{2}$, and
	$
	A_{G,uv}(0)=B_0>0.
$
	Consequently, at every fixed total weight in Case~3, the chosen
minimizing vector satisfies the strict form of
	\eqref{eq:inequality}.

The polynomial \eqref{eq:certificate-poly} is invariant under
exchanging $a,b$ and permuting $c_1,\ldots,c_{d-2}$.
Thus it depends only on the sorted endpoint degrees and the sorted
remaining degrees.
Appendix proves that its enumeration includes
every such degree signature and records the results for all
$4\le d\le25$.
This completes the proof.

\end{proof}

\begin{proof}[Proof of Theorem~\ref{thm:main}]
Fix $4\le d\le 25$.
Theorem~\ref{lem:critical} gives $\zeta_{L_d}=0$.
It remains to verify \eqref{eq-link-rigidity}.

Let $H$ be an $n$-vertex $3$-graph with $\Ex(n,L_d)+q$ edges
and $\N_{L_d}(H)=h_{L_d}(n,q)$. Take a link vector $\boldsymbol w$ satisfying
$R_{L_d}(\boldsymbol w)\le
\alpha_{L_d}\max\{\eta(\boldsymbol w),0\}$.
Suppose first that $\eta(\boldsymbol w)\le0$, then
$R_{L_d}(\boldsymbol w)=0$ since every term in \eqref{eq-link-polynomial} is nonnegative.
On the other hand, both terms on the right of \eqref{eq:counts} are nonnegative.
Hence
$
 \sum_iw_{ii}=0$ and
 $P_d(W_1+w_{11},\ldots,W_d+w_{dd})=0$.
Similarly, by \eqref{eq:polynomial},  there is  $W_i+w_{ii}=0$ for some $i$.

Now suppose that $\eta(\boldsymbol w)>0$.
By \eqref{eq:difference},
\[
 P_d(W_1+w_{11},\ldots,W_d+w_{dd})
 \le(d-2)^{3(d-1)}
       \left(\sum_{i<j}w_{ij}-\binom{d-1}{2}\right).
\]
If $\sum_{i<j}w_{ij}<\binom{d-1}{2}$, the right side is negative.
This is impossible.
If $\sum_{i<j}w_{ij}>\binom{d-1}{2}$, then
\[
 P_d(W_1+w_{11},\ldots,W_d+w_{dd})
 \ge P_d(W_1,\ldots,W_d)
 >(d-2)^{3(d-1)}
       \left(\sum_{i<j}w_{ij}-\binom{d-1}{2}\right).
\]
The first inequality uses the nonnegative coefficients of $P_d$, and the
 second is Theorem~\ref{thm:inequality}.
This is  also contradiction.
Therefore $\sum_{i<j}w_{ij}=\binom{d-1}{2}$,
and then $P_d(W_1+w_{11},\ldots,W_d+w_{dd})=0$.
Again $W_i+w_{ii}=0$ for some $i$.

In both cases,   $W_i=0, w_{ii}=0$ for some $i$.
Thus $W_i+w_{ii}/2=0$, as required by \eqref{eq-link-rigidity}.
Apply Theorem~\ref{thm:criterion}.
It gives  $T^{(3)}(n,d)\subseteq H$, and then $h_{L_d}(n,q)=\N_{L_d}(H)\ge t_{L_d}(n,q)$.
 The definition gives $h_{L_d}(n,q)\le t_{L_d}(n,q)$.
\end{proof}

\subsection{A family with threshold of order $n^2$}\label{sec:quadratic}
Throughout this section, fix $d\ge4$. Let $L=\K_3(F_{2,5,d})$,
and write $l=|V(L)|$.
The proof changes the part sizes and then adds a random set of triples.
We count copies with one, two, and at least three added edges separately.

Start with the balanced complete $d$-partite 3-graph $T^{(3)}(n,d)$, and  let $V(T^{(3)}(n,d))=V_1\cup\cdots\cup V_d$.
Let $\mathcal A_1$ be the triples with two vertices in  $V_1$
and one vertex outside it. Then
\begin{equation}\label{eq:pool-size}
 |\mathcal A_1|
 =\binom{|V_1|}{2}(n-|V_1|)
 =\frac{d-1}{2d^3}n^3+O_d(n^2).
\end{equation}
We count copies in the graph $T^{(3)}(n,d)\cup \mathcal A_1$ containing all crossing triples
and all triples of $\mathcal A_1$.
\begin{lem}\label{two-edge-count}
The number of copies of $L$ with exactly two edges of $\mathcal A_1$ is $\frac{5(d-1)^2}{d^3(d-2)}\alpha_L n^l+O_L(n^{l-1}).$
\end{lem}

\begin{proof}
Consider a copy $L'$ that uses at most two triples of $\mathcal A_1$.
Let $U_1=\{u,v\},U_2,\ldots,U_d$ be the original parts of the copy,
where $|U_j|=5$ for $j\ge2$.
Color each vertex of $L'$ by the index of its host part $V_i$.
If an original edge is monochromatic, its expanded triple is
non-crossing and must belong to $\mathcal A_1$.
Its two endpoints therefore have color $1$.
Since distinct original edges give distinct expanded triples,
there are at most two monochromatic original edges.

A color other than $1$ cannot occur in two original parts,
since vertices in different original parts are adjacent.
Suppose color $1$ occurs in two original parts.
The other $d-1$ colors are each confined to one original part.
Thus at least one of the $d$ original parts has only color $1$.
If this part has size five, its edges to a color-$1$ vertex in
another part give five monochromatic edges.
If it is $U_1=\{u,v\}$, its internal edge and its two edges to
such a vertex give three monochromatic edges.
Both conclusions contradict the bound of two.
Hence every color is confined to one original part.
There are $d$ nonempty original parts and only $d$ colors,
so each original part has one color.
The edge $uv$ forces the color of $U_1$ to be $1$.
Its expansion therefore belongs to $\mathcal A_1$.
Consider any other original edge.
If it joins $U_1$ to a part of color $j\ne1$, its expansion
is crossing when the new vertex has any of the $d-2$ colors
outside $\{1,j\}$.
It belongs to $\mathcal A_1$ precisely when the new vertex
has color $1$.
There are
\[
|U_1|\sum_{j=2}^d|U_j|=2\cdot5(d-1)=10(d-1)
\]
such original edges.
An edge between two other original parts has no endpoint
of color $1$.
Its expansion therefore cannot contain two vertices of color $1$,
so it cannot belong to $\mathcal A_1$.

The number of copies using just the expansion of $uv$ as an
added edge is
$\frac{d-1}{2d^3}\alpha_L n^l+O_L(n^{l-1}).
$
Indeed, there are \eqref{eq:pool-size} choices for that triple,
and each has single-edge count
$\alpha_Ln^{l-3}+O_L(n^{l-4})$.
At equal part sizes, each allowed color assignment has the same
leading contribution. To obtain exactly two added edges, choose
one of the $10(d-1)$ edges described above and replace its
$d-2$ crossing color choices by the one additional choice.
Thus the number of copies with exactly two added edges is
$
 \frac{5(d-1)^2}{d^3(d-2)}\alpha_L n^l+O_L(n^{l-1}).$
This argument may first be made for labelled copies.
Division by $|\Aut(L)|$ is the same for both counts.
\end{proof}

Note that the same leading term holds if the part sizes change by a fixed
number of vertices. For each fixed color assignment, the count
is a product of falling factorials in the part sizes.
A bounded change in the sizes changes a degree-$l$ term by
$O(n^{l-1})$. The constant in this error may depend on the
number of moved vertices.

\begin{proof}[Proof of Theorem~\ref{thm:quadratic}]
The lower bound $q_L(n)=\Omega_L(n^2)$ follows from
Theorem~\ref{thm1}.
By Theorem~\ref{lem:critical},
\[
 \zeta_L=d\alpha_L\bigl((2-5)(1-5)-2\bigr)=10d\alpha_L.
\]
We shall construct a graph $H$ with $\Ex(n,L)+q$ edges and fewer than
$qc(n,L)$ copies for some $q=O_L(n^2)$.

Let $A=q/n^2$, and let $a=\left\lfloor\frac{10d(d-1)}{d-2}A\right\rfloor.$
Here $a$ is the number of vertices to be moved.
For fixed $A$, both $a$ and $d$ are independent of $n$.
Remove $a$ vertices from  $V_1$ of a balanced partition. The move changes the first part $V_1$ by $-a$.
It changes the other parts by nonnegative integers whose
sum is $a$ and whose values differ by at most one.
Call the new complete $d$-partite $3$-graph $K'$, and let
$
 \tau=e(T^{(3)}(n,d))-e(K').
$
Thus $\tau\ge0$.
If the part sizes are $n/d+y_i$ with $\sum_i y_i=0$, then
\[
 e(K^{(3)}(V_1,\ldots,V_d))
 =\binom d3(n/d)^3
 -\frac{d-2}{2d}n\sum_i y_i^2
 +\sum_{i<j<h}y_iy_jy_h.
\]The increases in the other $d-1$ part sizes have squared sum
 $a^2/(d-1)+O_d(1)$.
It follows that the change in $\sum_i y_i^2$ is
$a^2+a^2/(d-1)+O_d(a+1)$.
The cubic terms change by $O_d((a+1)^3)$.
Consequently,
\begin{equation}\label{eq:crossing-loss}
 \tau=\frac{d-2}{2(d-1)}a^2n
       +O_d((a+1)n+(a+1)^3).
\end{equation} Choose a uniform $(q+\tau)$-subset of the new $\mathcal A_1$.
Add it to $K'$, and call the resulting graph $H$.
For fixed $L,A$, we have $q+\tau=O_{L,A}(n^2)$, whereas
$|\mathcal A_1|=\Theta_L(n^3)$.
Thus this choice is possible for large $n$.
The extremal theorem  \ref{extremal-graph} gives
\[
 e(H)=e(K')+q+\tau=e(T^{(3)}(n,d))+q=\Ex(n,L)+q.
\]

We also need the change in the single-edge count.
In the original balanced partition, any added triple in $\mathcal A_1$ has
count $c(n,L)+O_L(n^{l-4})$.
This follows from Lemma~\ref{lem-cs}, since all simple
patterns have leading coefficient $\alpha_L$.
Thus, uniformly over all triples in the new $\mathcal A_1$,
\begin{align}\label{eq:single-edge-change}
 \N_L(K'\cup\{e\})
 =&c(n,L)-a\zeta_Ln^{l-4}+O_L\bigl(n^{l-4}+(a+1)^2n^{l-5}\bigr)\notag\\
 =
 &c(n,L)-
 10d\alpha_La n^{l-4}
  +O_L\bigl(n^{l-4}+(a+1)^2n^{l-5}\bigr).
\end{align}
Each selected triple contributes the copies using it as their
only added edge. These families are disjoint.
Summing \eqref{eq:single-edge-change}, then subtracting $qc(n,L)$,
gives
\[
 \tau c(n,L)-10d\alpha_La(q+\tau)n^{l-4}
 +O_L\bigl((q+\tau)n^{l-4}
           +(q+\tau)(a+1)^2n^{l-5}\bigr).
\]
Use \eqref{eq:crossing-loss} and
$c(n,L)=\alpha_Ln^{l-3}+O_L(n^{l-4})$.
Since $a=O_L(A)$, this expression equals
\begin{equation}\label{eq:one-edge-comparison}
 \alpha_L\left(\frac{d-2}{2(d-1)}a^2-10dAa\right)n^{l-2}
 +O_L(A+1)n^{l-2}+O_{L,A}(n^{l-3}).
\end{equation}
The constant in $O_L(A+1)$ does not depend on $A$.
The terms with higher powers of $a$ are included in
$O_{L,A}(n^{l-3})$.

A fixed copy using $j$ distinct added triples is present with
probability
$
 \frac{(q+\tau)_j}{(|\mathcal A_1|)_j}
 \le\left(\frac{q+\tau}{|\mathcal A_1|}\right)^j.
$
For $j=2$, use Lemma~\ref{two-edge-count} in the new partition.
Together with \eqref{eq:pool-size}, it bounds the expected count by
\begin{align}
 &\left(\frac{5(d-1)^2}{d^3(d-2)}\alpha_L n^l
             +O_{L,A}(n^{l-1})\right)
       \left(\frac{q+\tau}{|\mathcal A_1|}\right)^2
 =\frac{20d^3}{d-2}\alpha_L A^2n^{l-2}
       +O_{L,A}(n^{l-3}).
 \label{eq:beta}
\end{align}
There are $O_L(n^l)$ potential copies in total.
For any copy using at least three added triples, the probability
that three prescribed ones are all selected is $O_{L,A}(n^{-3})$.
Thus all copies using at least three added triples have expected
count $O_{L,A}(n^{l-3})$.
Combine these estimates with \eqref{eq:one-edge-comparison} and \eqref{eq:beta}.
The coefficient of $\alpha_Ln^{l-2}$, apart from the
$O_L(A+1)$ term, is
\begin{align*}
 \frac{d-2}{2(d-1)}a^2&-10dAa+\frac{20d^3}{d-2}A^2\\=
 &\frac{d-2}{2(d-1)}
    \left(a-\frac{10d(d-1)}{d-2}A\right)^2
 -\frac{10d^2(3d-5)}{d-2}A^2.
\end{align*}
By the choice of $a$, we have
\begin{equation}\label{eq:comparison}
 \begin{split}
 \EE(\N_L(H))-qc(n,L)
 \le{}&
 \left(-\frac{10d^2(3d-5)}{d-2}\alpha_L A^2
             +O_L(A+1)\right)n^{l-2}\\
 &+O_{L,A}(n^{l-3}).
 \end{split}
\end{equation}
The coefficient of $A^2$ is negative for every $d\ge4$.
Then for sufficient large  $n$, choose $A_d$ large enough that the right side of \eqref{eq:comparison} is negative for all $A\ge A_d$.
Some choice of $H$ therefore has fewer than $qc(n,L)$ copies.
Our construction gives $h_L(n,q)< qc(n,L)\le t_L(n,q)$.
Thus $q_L(n)<q=O_L(n^2)$.
Together with the cited lower bound, this proves the theorem.
\end{proof}

\appendix
\section{Proof of the extremal theorem}\label{A}
	
	\begin{proof}[Proof of Theorem~\ref{extremal-graph}]
		Suppose that $n$ is large enough.	Fix a sequence of positive constants
		\[
		c_1 \gg c_2 \gg c_3 \gg c_4 \gg c_5,
		\]
		each chosen sufficiently large relative to the preceding ones.
		Let $L\in \mathcal{F}(k,d)$ be a $k$-graph with $L=\mathcal{K}^{(k)}(F)$, and let $f:=|V(F)|$,  $m:=|E(F)|$ and  $l:=|V(L)|=f+m(k-2)$.
		Let  $G$ be a maximum $L$-free $k$-graph on $n$ vertices. 
		Since $T^{(k)}(n,d)$ is $L$-free, by definition,
		$
		e(G) \ge e(T^{(k)}(n,d)).
		$
		Since $n$ is sufficiently large, Theorem~\ref{thm-stable} implies that $G$ is $c_5 n^k$-close to $T^{(k)}(n,d)$.	Partition $V(G)$ into $V_1 \cup \cdots \cup V_d$ so that the number $h =\sum_{i\in[k]}ih_i$ 
		is maximized, where	$h_i$ be the number of edges of $G$ intersecting exactly $i$ of the parts from $V_1 \cup \cdots \cup V_d$.
		Let $\prod_{i=1}^dV_i$ be the complete $d$-partite $k$-graph with parts $V_1,\ldots,V_d$. Then clearly $h \ge k|E(G)\cap E(\prod_{i=1}^dV_i) |$.	
		By the maximality of $h$, we also have
		$
		h \ge k(|E(G)| - c_5 n^k),
		$
		while on the other hand,
		$
		h \le k|E(G)| - |E(G) \setminus E(\prod_{i=1}^dV_i)|.
		$
		Combining these inequalities, we obtain
		$
		|E(G) \setminus E(\prod_{i=1}^dV_i)| \le c_5 k n^k.
		$
		
		Define the set of missing edges $M = E\big(\prod_{i=1}^d V_i\big) \setminus E(G)$ and the set of bad edges $B = E(G) \setminus E\big(\prod_{i=1}^d V_i\big)$, so $|B|\le c_5 k n^k$. We claim $|M| \le |B|$, otherwise 
		$e(G)= e(\prod_{i=1}^dV_i)+|B|-|M|< e(\prod_{i=1}^dV_i)$, a	contradiction. If $M=\emptyset$ and $B=\emptyset$, then  
		$G =\prod_{i=1}^dV_i$, we are done since $T^{(k)}(n,d)$ is  the complete $d$-partite $k$-graph with the largest number of edges. 
		Next we assume that $|B|\ge1$.	
		Since 	$e(\prod_{i=1}^dV_i) \ge e(T^{(k)}(n,l)) - c_5 k n^k$. 
		By  Lemma \ref{lem-st}, it follows that each part $V_i$ satisfies
		$
		\frac{n}{2d}\le |V_i| \le \frac{2n}{d}.
		$
		
		For pairs of vertices $x,y$ from different parts $V_i$, call $\{x,y\}$ \emph{sparse} if $d_G(xy)\le ln^{k-3}$, and \emph{dense} otherwise.
		Note that there are fewer than $c_4 n^2$ sparse pairs.  Otherwise, these sparse pairs produce at least  
		\begin{align*}
		c_4 n^2\left( \left( {n\over 2d}\right) ^{k-2}-ln^{k-3}\right) \ge c_4\left( {n\over 2d}\right) ^{k} 
		\end{align*}
		missing edges.  Since each missing edge can be generated by at most  ${k\choose 2}$ sparse pairs, we derive that $|M|\ge  c_4 ({n\over 2d})^{k}/{k\choose 2},$
		leading to a contradiction with $|M^*|\le c_5 k n^k$.
		
		We call a vertex pair from a bad edge a \emph{bad pair}  if  it is contained in  same partition class. 
		\begin{clm}\label{clm-bad}
			Let
			$A := \big\{ v \in V(G) \mid \text{the number of sparse pairs containing } v \text{ is at least } c_1 n \big\}$, then we have every bad pair is intersects $A$.	
		\end{clm}
		Choose a bad edge
		$
		e\in B
		$. 
		Take an arbitrary bad pair $\{x_1,x_2\}$ from $e$, where, without loss of generality,
		$
		x_1,x_2\in V_1
		$.
		Now choose vertex set $
		U_i\in V_i\setminus e$ for each $
		1\le i\le d.
		$
		Observe that, if every pair $\{a,b\}$ where
		$
		a\in \{x_1,x_2\}\cup(\bigcup_{i=1}^d U_i)
		$ and $b\in
		\bigcup_{i=1}^d U_i
		$
		is dense. 
		Then, proceeding greedily, we can always select an edge of $G$ covering this pair and avoiding all previously chosen vertices in a partial embedding.
		Conclusion, this extend the bad edge $e$ and obtain a copy of $L$ in $G$ with original vertex set $\{x_1,x_2\}\cup(\bigcup_{i=1}^d U_i)$, contradicting the assumption that $G$ is $L$-free.
		Therefore, for every choice of
		$$
		U_1\times\cdots\times U_d\subseteq
		(V_1\setminus e)\times\cdots\times(V_d\setminus e),
		$$
		there must exist a sparse pair $\{u,v\}\neq\{x_1,x_2\}$  satisfying $
		u\in \{x_1,x_2\}\cup(\bigcup_{i=1}^d U_i)
		$  and $v\in
		\bigcup_{i=1}^d U_i
		$.
		Using
		$
		|V_i|
		\ge
		\frac{n}{2d}
		$, we have the total number of such choice of
		$
		U_1\times\cdots\times U_d$ is at least
		$$\prod_{i=1}^d{\frac{n}{2d}-k\choose |U_i|}
		\ge {1\over 2f!}\left(\frac{n}{2d}\right)^{\sum_{i\in[d]}|U|_i}={1\over 2f!}\left(\frac{n}{2d}\right)^{f-2}.
		$$
		On the other hand, 
		if the sparse pair is disjoint from
		$
		\{x_1,x_2\},
		$
		then after fixing this pair there remain at most
		$
		n^{l-4}
		$
		ways to choose the other vertices. Recall that we have less than $c_4n^2 $ sparse pair, by choosing $c_4$ sufficiently small 
		we obtain
		$c_4n^2 \cdot n^{f-4}<{1\over 4f!}\left(\frac{n}{2d}\right)^{f-2}.
		$ It follows that at least half of the 
		all choices of
		$
		U_1\times\cdots\times U_d$, the resulting sparse pair intersects  $\{x_1,x_2\}$. 
		Since each one of such  sparse pair,
		it can arise from at most
		$
		n^{l-3}
		$
		choices. It follows that at least 
		\[
		{1\over 4f!}\left(\frac{n}{2d}\right)^{f-2}
		n^{-(f-3)}
		\ge 2c_1n,
		\] sparse pair intersects  $\{x_1,x_2\}$, and then
		one of the vertices of $x_1,x_2$ belongs to $A$.	
		This completes the proof of the Claim.

		We now count the missing edges incident with vertices of $A$. Note that $A$ intersects at least ${|A|c_1n\over 2}$ sparse pairs, and every sparse pair generates at least
		$
		\left( 
		\frac{n}{2d}
		\right)^{k-2}-ln^{k-3}
		$
		missing edges. Furthermore,  every missing edge contains at most
		$
		\binom{k}{2}
		$
		sparse pairs.
		Consequently, the number of missing edges is at least
		\begin{align*}
		{|A| c_1n\over 2}\cdot \left( \left(
		\frac{n}{2d}
		\right)^{k-2}-ln^{k-3}\right) \cdot\binom{k}{2}^{-1}\ge 	|A|\cdot c_2 n^{k-1}.
		\end{align*}
		So $|B|\ge|M|\ge |A|\cdot c_2 n^{k-1}$.	Let	$\mathcal B=\{(e,\{x,y\}):e\in B,\, x,y\in e ,\, \{x,y\} \ is\ a\ bad\ pair
		\}$, then $|\mathcal B|
		\ge
		|A|c_2n^{k-1}
		$,
		since each bad edge contains a bad pair.
		On the other hand, every bad pair intersects $A$. Let $B(x)$ denote the set of  edges in $B$ that contain vertex $x$ and  another vertex from the same  partition class as $x$.
		Fixing a vertex $x$ and an edge $e\in B(x)$, there are at most	$k-1$
		choices for a bad pair contained in $e$ and containing $x$. 
		Therefore, $|\mathcal B|\le \sum_{x\in A}(k-1)|B(x)|$. As shown above, there exists a vertex
		$
		x\in A,
		$
		say
		$
		x\in V_1,
		$ such that	$|B(x)|\ge \frac{c_2}{k-1}n^{k-1}$.
		
		Altogether, we conclude from the above discussion that there exists a vertex $x\in V_1$ belongs to at least  $\frac{c_2}{k-1}n^{k-1}$ bad edges, each containing another vertex from $V_1$.
		Define
		$
		T = N_B(x) \cap V_1
		$
		to be the set of vertices in $V_1$ adjacent to $x$ via a bad pair.	 Then
		$|T|	\ge \frac{c_2}{k-1}n^{k-1}\cdot \binom{n}{k-2}^{-1}
		\ge c_3n.
		$
		For every
		$i\in\{2,\ldots,d\}
		$,
		define
		$
		Y_i
		=
		\{y\in V_i:\{x,y\}\ \text{is dense}\}.
		$
		For every vertex
		$
		y\in Y,
		$
		we can find  an edge
		$
		e_y\in E(G)
		$
		containing both $x$ and $y$.
		Next choose vertex sets $U_1\in Y_1\setminus e_y$ and 
		$
		U_i\in Y_i\setminus e_y
		$.
		Using the density of all pairs $\{x,u\}$ with  $u\in \bigcup_{i=2}^dU_i$, together with the edge $e_y$, we can greedily construct a partial embedding of $L$  whose every pair containing $x$ is already covered. 
		Since $G$ is $L$-free and every pair $\{x,v\}$ with $v\in \{y\}\cup U_1$ is not an edge in $F$, at least one of the pairs in  $\{y\}\cup (\bigcup_{i=2}^dU_i)$ must be sparse.	
		If 	$|Y_j|\ge c_3n$ for all
		$
		j=2,\ldots,d.
		$
		Then the total number of such choices of  $\{y\} \cup (\bigcup_{i=2}^dU_i)$ is at least
		$$c_3n\cdot {{n\over 2d}-k\choose |U_1|} \prod_{i=2}^d{c_3n-k\choose |U_i|}>c_4n^{f-1}.
		$$
		Note that $x$ is fixed, hence we obtain more than	$c_4n^{f-1}/n^{f-3}$ sparse pairs, contradicting the previously established bound.
		Hence, there exists some
		$
		j\in\{2,\ldots,d\}
		$
		for which
		$
		|Y_j|<c_3n.
		$
		By symmetry, we may assume that
		$
		|Y_2|<c_3n
		$, that is all but at most $c_3n$ pairs  $
		\{x,u\}$ with $ u\in V_2
		$
		are dense.
		Let $B_1(x)$ denote 
		the number of edges of $G$ containing $x$ and intersecting $V_2$, then 
		$$|B_1(x)|\le 
		c_3n\binom{n}{k-2}+n\cdot ln^{k-3}
		< c_3n^{k-1}.
		$$
		Let $B_2(x)$ denote 
		the number of bad edges of $B$ 
		contain $x$, another vertex from $V_1$, and no vertex from $V_2$, then
		\[|B_2(x)|\ge |B(x)|-|B_1(x)|\ge
		\frac{c_2}{k-1}n^{k-1}-c_3n^{k-1}>|B_1(x)|
		.\]

		We now consider moving the vertex $x$ from the part $V_1$ to the part $V_2$.
		Recall that the partition
		$
		V(G)=V_1\cup\cdots\cup V_d
		$
		was chosen so that
		$
		h=\sum_{i\in[d]}ih_i
		$
		is maximized.
		After relocating $x$ to $V_2$, some edges
		may contribute less to $h$.
		However, this can happen only for edges
		of $B_1(x)$.
		On the other hand, each  edge of $B_2(x)$ contribute one additional part, and therefore increases the value of $h$ by exactly one.
		Thus the total gain in $h$ is at least
		$|B_2(x)|-|B_1(x)|>0
		$, this contradicts the maximality of the original partition
		$
		V_1\cup\cdots\cup V_d.
		$
	\end{proof}
\section{Exact verification of the weighted inequality}

The proof of Theorem~4.1 reduces the finite check to
graphs $G$ with $\binom{d-1}{2}\le e(G)<\binom d2$ of weight-one edges and a distinguished
missing edge $uv$  assigned weight $t\in[0,1)$.
If $a,b$ are the degrees of $u,v$ in $G$ and
$c_1,\ldots,c_{d-2}$ are the other degrees, the weighted degrees are
\[
a+t,\ b+t,\ c_1,\ldots,c_{d-2}.
\]
The polynomial in (12) depends only on the unordered pair $\{a,b\}$
and the multiset $\{c_1,\ldots,c_{d-2}\}$.
We call these two sorted lists the \emph{degree signature}.
The program enumerates degree signatures directly.

\subsection*{Complete enumeration of the degree signatures}

Let $J=K_d-G$ be the complement of $G$.
Then $uv\in E(J)$ and
$1\le e(J)\le d-1.$
For each possible value $h=e(J)$, the positive degrees of $J$
form an integer partition of $2h$ into at most $d$ parts.
Every part is at most $h$, since every vertex has degree at most \(h\).
The program enumerates all these partitions, adds zero degrees
to obtain $d$ entries, and tests whether the sequence is graphical.
A sequence is \emph{graphical} if it is the degree sequence of
a simple graph.

For each graphical sequence, choose the degrees of the two
endpoints of the distinguished edge.
They must be positive.
A repeated value can be chosen twice only if it occurs at least twice.
It remains to test whether some realization of the sequence
contains an edge between these two specified vertices.

Suppose the chosen endpoint degrees in $J$ are $a,b$, and the
other degrees, in nonincreasing order, are
$c_1,\ldots,c_{d-2}$.
There is  containing $uv$ if and only if
\[
b-1,\quad
c_1-\mathbf1_{\{1\le a-1\}},\quad\ldots,\quad
c_{d-2}-\mathbf1_{\{d-2\le a-1\}}
\]
is graphical and  nonnegative.
Here $\mathbf1_{\{j\le a-1\}}$ is one when $j\le a-1$ and zero otherwise.
We prove this equivalence.
Delete $uv$ from a realization of $J$.
The remaining degree of $u$ is $a-1$, and $u$ is not adjacent to $v$.
Let $S$ consist of the $a-1$ other vertices of largest prescribed degree.
We may choose a realization with $N(u)=S$.
Indeed, choose a realization maximizing $|N(u)\cap S|$.
If $N(u)\ne S$, take $x\in S\setminus N(u)$ and
$y\in N(u)\setminus S$.
Then $\deg_J(x)\ge\deg_J(y)$.
Note that $N(y)\setminus\{x\}$ contains $u$, whereas
$N(x)\setminus\{y\}$ does not.
Since the latter set is at least as large, there is
\[
z\in\bigl(N(x)\setminus\{y\}\bigr)
\setminus\bigl(N(y)\setminus\{x\}\bigr).
\]
The four vertices $u,x,y,z$ are distinct.
Replace the edges $uy,xz$ by $ux,yz$.
This increases $|N(u)\cap S|$, a contradiction.
Deleting $u$ now gives the  sequence above.
Conversely, any  that  sequence can be
extended by adding $u$, joining it to $S$, and restoring $uv$.
This proves the equivalence.

The routine \texttt{graphical} uses the Havel--Hakimi reduction.
It deletes a largest degree $r$ and subtracts one from the next
$r$ largest degrees.
It repeats this operation until all degrees are zero or a
negative or impossible degree is encountered.
The same exchange argument, without an excluded vertex, proves
that each reduction preserves graphicality.
Thus both graphicality tests are exact.

Every admissible pair $(G,uv)$ produces one of the partitions
and one of the endpoint choices above.
Conversely, every choice that passes the graphicality test has a realization containing
the distinguished edge in $J$.
Taking complements gives an admissible pair $(G,uv)$.
Hence the enumeration includes exactly the degree signatures
needed in the concavity reduction.
This proves completeness without enumerating every labelled graph.

The complete R program  can be obtained by contacting the author if needed.

\end{document}